\documentclass[10pt]{amsart}

\usepackage{amsthm,amsfonts,amsmath,verbatim,mathrsfs,amssymb,verbatim,cite}
\usepackage[usenames]{color}
\usepackage{tikz,graphicx}
\usepackage[backref]{hyperref}

\numberwithin{equation}{section}
\newtheorem{theorem}{Theorem}[section]
\newtheorem{corollary}[theorem]{Corollary}
\newtheorem{proposition}[theorem]{Proposition}

\newtheorem{conjecture}[theorem]{Conjecture}
\newtheorem{lemma}[theorem]{Lemma}

\theoremstyle{remark}
\newtheorem{remark}[theorem]{Remark}

\newcommand{\hyper}[5]{{_{#1}F_{#2}}\bigg[\genfrac{}{}{0pt}{}{#3}{#4};#5\bigg]}

\newcommand{\ph}{\hphantom{-}}
\newcommand{\Tau}{\textrm{T}}

\newcommand{\Int}{\int\limits}
\renewcommand{\Re}{\textup{Re}}
\renewcommand{\and}{\quad\text{and}\quad}
\DeclareMathOperator*{\Pf}{Pf}
\DeclareMathOperator*{\Pfb}{Pf\!}

\DeclareMathOperator{\sgn}{sgn}

\DeclareMathOperator{\CT}{CT}
\DeclareMathOperator{\eup}{e}
\DeclareMathOperator{\height}{ht}
\DeclareMathOperator{\iup}{i}
\DeclareMathOperator{\dup}{d\hspace{-1.5pt}}

\newcommand{\Real}{\mathbb R}

\newcommand{\Z}{\mathbb Z}

\newcommand{\Complex}{\mathbb C}
\newcommand{\Symm}{\mathfrak{S}}
\renewcommand{\L}{\mathscr{L}}
\renewcommand{\L}{L}
\newcommand{\abs}[1]{\lvert#1\rvert}

\newcommand{\ib}{i}
\newcommand{\jb}{j}
\newcommand{\R}{\Phi}
\newcommand{\bm}{m}
\newcommand{\bb}{b}
\newcommand{\fb}{a}
\newcommand{\eb}{e}
\newcommand{\eez}{\varnothing}
\newcommand{\een}{1\kern-2.75pt\mathrm{I}}
\newcommand{\eem}{I}

\makeatletter
\renewcommand*\env@matrix[1][c]{\hskip -\arraycolsep
  \let\@ifnextchar\new@ifnextchar
  \array{*\c@MaxMatrixCols #1}}
\makeatother

\begin{document}

\title{Logarithmic~and~complex constant~term~identities}

\author{Tom Chappell}
\address{School of Mathematics and Physics,
The University of Queensland, Brisbane, QLD 4072, Australia}

\author{Alain Lascoux}
\address{CNRS, Institut Gaspard Monge,
Universit\'e Paris-Est, Marne-la-Vall\'ee, France}

\author{S.~Ole Warnaar}
\address{School of Mathematics and Physics,
The University of Queensland, Brisbane, QLD 4072, Australia}

\author{Wadim Zudilin}
\address{School of Mathematical and Physical Sciences,
The University of Newcastle, Call\-aghan, NSW 2308, Australia}

\thanks{Ole Warnaar and Wadim Zudilin are supported by the Australian Research Council}

\dedicatory{To Jon}


\begin{abstract}
In recent work on the representation theory of vertex algebras
related to the Virasoro minimal models $M(2,p)$,
Adamovi\'c and Milas discovered logarithmic analogues of (special cases
of) the famous Dyson and Morris constant term identities.
In this paper we show how the identities
of Adamovi\'c and Milas arise naturally by
differentiating as-yet-conjectural complex analogues of
the constant term identities of Dyson and Morris.
We also discuss the existence of complex and logarithmic
constant term identities for arbitrary root systems, and in particular
prove complex and logarithmic constant term identities
for the root system $\mathrm{G}_2$.

\end{abstract}

\maketitle

\setcounter{section}{-1}

\section{Jonathan Borwein}
Jon Borwein is known for his love of mathematical \textit{constants}.
We hope this paper will spark his interest in \textit{constant terms}.

\section{Constant term identities}

The study of constant term identities originated in Dyson's famous 1962
paper \emph{Statistical theory of the energy levels of complex systems}
\cite{Dyson62}. In this paper Dyson conjectured that for $a_1,\dots,a_n$
nonnegative integers
\begin{equation}\label{Eq_CT_Dyson}
\CT
\prod_{1\leq i\neq j\leq n} \Big(1-\frac{x_i}{x_j}\Big)^{a_i}
=\frac{(a_1+a_2+\cdots+a_n)!}{a_1!a_2!\cdots a_n!},
\end{equation}
where $\CT f(X)$ stands for the constant term of the Laurent polynomial
(or possibly Laurent series) $f(X)=f(x_1,\dots,x_n)$.
Dyson's conjecture was almost instantly proved by Gunson and
Wilson \cite{Gunson,Wilson62}. In a very elegant proof, published
several years later \cite{Good70}, Good showed that
\eqref{Eq_CT_Dyson} is a direct consequence of Lagrange interpolation
applied to $f(X)=1$.

In 1982 Macdonald generalised the equal-parameter case of Dyson's
ex-conjecture, i.e.,
\begin{equation}\label{Eq_CT_A}
\CT
\prod_{1\leq i\neq j\leq n} \Big(1-\frac{x_i}{x_j}\Big)^k
=\frac{(kn)!}{(k!)^n},
\end{equation}
to all irreducible, reduced root systems;
here \eqref{Eq_CT_A} corresponds to
the root system $\mathrm{A}_{n-1}$.
Adopting standard notation and terminology---see
\cite{Humphreys78}
or the next section---Macdonald conjectured that \cite{Macdonald82}
\begin{equation}\label{Eq_CT_Macdonald}
\CT
\prod_{\alpha\in\R} (1-\eup^{\alpha})^k=\prod_{i=1}^r \binom{k d_i}{k},
\end{equation}
where $\Phi$ is one of the root systems
$\mathrm{A}_{n-1},\mathrm{B}_n,\mathrm{C}_n,\mathrm{D}_n,\mathrm{E}_6,
\mathrm{E}_7,\mathrm{E}_8,\mathrm{F}_4,\mathrm{G}_2$ of rank $r$ and
$d_1,\dots,d_r$ are the degrees of its fundamental invariants.
For $k=1$ the Macdonald conjectures are an easy consequence of Weyl's
denominator formula
\[
\sum_{w\in W}\sgn(w) \eup^{w(\rho)-\rho}=\prod_{\alpha>0}
\big(1-\eup^{-\alpha}\big)
\]
(where $W$ is the Weyl group of $\Phi$ and $\rho$ the Weyl vector), and
for $\mathrm{B}_n,\mathrm{C}_n,\mathrm{D}_n$ but $k$ general
they follow from the Selberg integral.
The first uniform proof of \eqref{Eq_CT_Macdonald}---based on hypergeometric
shift operators---was given by Opdam in 1989 \cite{Opdam89}.

In his PhD thesis \cite{Morris82} Morris used the Selberg integral
to prove a generalisation of \eqref{Eq_CT_A},
now commonly referred to as the Morris or Macdonald--Morris
constant term identity:
\begin{multline}\label{Eq_CT_Morris}
\CT\bigg[\,\prod_{i=1}^n (1-x_i)^a \Big(1-\frac{1}{x_i}\Big)^b
\prod_{1\leq i\neq j\leq n} \Big(1-\frac{x_i}{x_j}\Big)^k\bigg] \\
=\prod_{i=0}^{n-1} \frac{(a+b+ik)!((i+1)k)!}{(a+ik)! (b+ik)! k!},
\end{multline}
where $a$ and $b$ are arbitrary nonnegative integers.

\medskip

In their recent representation-theoretic work on $W$-algebra extensions
of the $M(2,p)$ minimal models of conformal field theory \cite{AM10,AM11},
Adamovi\'c and Milas discovered a novel type of constant term identities,
which they termed \emph{logarithmic constant term identities}. Before stating
the results of Adamovi\'c and Milas, some more notation is needed.

Let $(a)_m=a(a+1)\cdots(a+m-1)$ denote the usual Pochhammer symbol
or rising factorial, and let $u$ be either a formal or complex variable.
Then the (generalised) binomial coefficient $\binom{u}{m}$ is defined as
\[
\binom{u}{m}=(-1)^m \frac{(-u)_m}{m!}.
\]
We now interpret $(1-x)^u$ and $\log(1-x)$ as the (formal) power series
\begin{equation}\label{Eq_binomial_thm}
(1-x)^u=\sum_{m=0}^{\infty} (-x)^m \binom{u}{m}
\end{equation}
and
\[
\log(1-x)
=-\sum_{m=1}^{\infty} \frac{x^m}{m}
=\frac{\dup\;}{\dup u}(1-x)^u\Big|_{u=0}.
\]
Finally, for $X=(x_1,\dots,x_n)$ we define the Vandermonde product
\[
\Delta(X)=\prod_{1\leq i<j\leq n}(x_i-x_j).
\]

One of the discoveries of Adamovi\'c and Milas is the following
beautiful logarithmic analogue of the equal-parameter case
\eqref{Eq_CT_A} of Dyson's identity.
\begin{conjecture}[{\!\!\cite[Conjecture A.12]{AM10}}]\label{Conj_log_CT_AM}
For $n$ an odd positive integer and $k$ a nonnegative integer
define $m:=(n-1)/2$ and $K:=2k+1$. Then
\begin{equation}\label{Eq_log_CT_AM}
\CT\bigg[\,\Delta(X)
\prod_{i=1}^n x_i^{-m} \prod_{i=1}^m \log\Big(1-\frac{x_{2i}}{x_{2i-1}}\Big)
\prod_{1\leq i\neq j\leq n} \Big(1-\frac{x_i}{x_j}\Big)^k\bigg]
=\frac{(nK)!!}{n!!(K!!)^n}.
\end{equation}
\end{conjecture}
We remark that the kernel on the left is a Laurent series in $X$ of
(total) degree $0$. Moreover, in the absence of the term
$\prod_{i=1}^m \log(1-x_{2i}/x_{2i-1})$ the kernel is
a skew-symmetric Laurent polynomial which therefore has a
vanishing constant term.
Using identities for harmonic numbers, Adamovi\'c and Milas proved
\eqref{Eq_log_CT_AM} for $n=3$, see \cite[Corollary 11.11]{AM10}.

Another result of Adamovi\'c and Milas, first conjectured in
\cite[Conjecture 10.3]{AM10} (and proved for $n=3$ in
(the second) Theorem 1.1 of that paper, see page 3925)
and subsequently proved in \cite[Theorem 1.4]{AM11},
is the following Morris-type logarithmic constant term identity.
\begin{theorem}\label{Thm_log_CT_AM2}
With the same notation as above,
\begin{multline}\label{Eq_CT_AM}
\CT\bigg[\Delta(X)
\prod_{i=1}^n x_i^{2-(k+1)(n+1)} (1-x_i)^a \prod_{i=1}^m
\log\Big(1-\frac{x_{2i}}{x_{2i-1}}\Big)
\prod_{1\leq i\neq j\leq n} \Big(1-\frac{x_i}{x_j}\Big)^k\bigg] \\
=c_{nk}\prod_{i=0}^{n-1} \binom{a+Ki/2}{(m+1)K-1},
\end{multline}
where $a$ is an indeterminate, $c_{nk}$ a nonzero constant, and
\begin{equation}\label{Eq_c3k}
c_{3,k}=
\frac{(3K)!(k!)^3}{6(3k+1)!(K!)^3}
\binom{3K-1}{2K-1}^{-1}\binom{5K/2-1}{2K-1}^{-1}.
\end{equation}
\end{theorem}
As we shall see later, the above can be generalised to include
an additional free parameter resulting in a
logarithmic constant term identity more closely resembling
Morris' identity, see \eqref{Eq_log_CT_Morris} below.

\medskip

The work of Adamovi\'c and Milas raises the following obvious questions:
\begin{enumerate}
\item Can any of the methods of proof of the classical
constant term identities, see e.g.,
\cite{BG85,Cherednik95,GLX08,GX06,Good70,Gunson,Habsieger86,Kadell98,Kaneko02,KN11,Opdam89,Sills06,Sills08,SZ06,Wilson62,Zeilberger82,Zeilberger87,Zeilberger90,ZB90},
be utilised to prove the logarithmic counterparts?
\item Do more of Macdonald's identities \eqref{Eq_CT_Macdonald} admit
logarithmic analogues?
\item All of the classical constant term identities have
$q$-analogues \cite{Habsieger88,Kadell88,Macdonald82,Morris82}.
Do such $q$-analogues also exist in the logarithmic setting?
\end{enumerate}

As to the first and third questions, we can be disappointingly short;
we have not been able to successfully apply any of the known methods of
proof of constant term identities to also prove
Conjecture~\ref{Conj_log_CT_AM}, and
attempts to find $q$-analogues have been equally unsuccessful.
(In fact, we now believe $q$-analogues do not exist.)

As to the second question, we have found a very appealing
explanation---itself based on further conjectures!---of the logarithmic
constant term identities of Adamo\-vi\'c and Milas.
They arise by differentiating a complex version of Morris' constant term
identity. Although such complex constant term identities are conjectured
to exist for other root systems as well---this is actually proved in the
case $\mathrm{G}_2$---it seems that only for $\mathrm{A}_{2n}$ and
$\mathrm{G}_2$ these complex identities imply elegant logarithmic
identities.

\medskip

The remainder of this paper is organised as follows. In the next section
we introduce some standard notation related to root systems.
Then, in Section~\ref{Sec_tau}, we study certain sign functions
and prove a related Pfaffian identity needed subsequently.
In Section~\ref{Sec_Complex_Morris},
we conjecture a complex analogue of the
Morris constant term identity~\ref{Eq_CT_Morris} for $n$ odd, and prove
this for $n=3$ using Zeilberger's method of creative telescoping
\cite{AZ06,PWZ96}.
In Section~\ref{Sec_CT_Log_Morris} we show that the complex Morris
identity implies the
following logarithmic analogue of \eqref{Eq_CT_Morris}.
\begin{theorem}[\textbf{Logarithmic Morris constant term identity}]\label{Thm_log_CT_Morris}
With the same notation as in Conjecture~\ref{Conj_log_CT_AM} and
conditional on the complex Morris constant term identity
\eqref{Eq_complex_Morris} to hold, we have
\begin{multline}\label{Eq_log_CT_Morris}
\CT\bigg[
\Delta(X)
\prod_{i=1}^n x_i^{-m} (1-x_i)^a\Big(1-\frac{1}{x_i}\Big)^b
\prod_{i=1}^m \log\Big(1-\frac{x_{2i}}{x_{2i-1}}\Big)
\prod_{1\leq i\neq j\leq n} \Big(1-\frac{x_i}{x_j}\Big)^k\bigg] \\
=\frac{1}{n!!}
\prod_{i=0}^{n-1}
\frac{(2a+2b+iK)!!((i+1)K)!!}{(2a+iK)!!(2b+iK)!!K!!},
\end{multline}
where $a,b$ are nonnegative integers.
\end{theorem}

In Section~\ref{Sec_G2} we prove complex as well as logarithmic analogues of
\eqref{Eq_CT_Macdonald} for the root system $\mathrm{G}_2$,
and finally, in Section~\ref{Sec_Other} we briefly discuss the classical
roots systems $\mathrm{B}_n$, $\mathrm{C}_n$ and $\mathrm{D}_n$.

\section{Preliminaries on root systems and constant terms}\label{Sec_2}

In the final two sections of this paper we consider root systems of types
other than $\mathrm{A}$, and below we briefly recall some standard
notation concerning root systems and constant term identities. For more
details we refer the reader to \cite{Humphreys78,Macdonald82}.

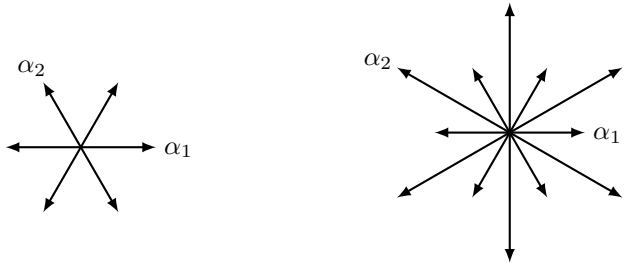
\begin{figure}[thb]
\begin{minipage}{0.4\linewidth}
\begin{center}
\begin{tikzpicture}[scale=0.5,baseline=0cm]
\draw[thick,-latex] (0,0) -- (2,0);
\draw[thick,latex-] (-2,0) -- (0,0);
\draw[thick,latex-latex] (-1,-1.73205)--(1,1.73205);
\draw[thick,-latex] (0,0)--(-1,1.73205);
\draw[thick,latex-] (1,-1.73205)--(0,0);
\draw (2.6,-0.07) node {$\alpha_1$};
\draw (-1.3,2.1) node {$\alpha_2$};
\end{tikzpicture}
\end{center}
\end{minipage}
\begin{minipage}{0.4\linewidth}
\begin{center}
\begin{tikzpicture}[scale=0.5,baseline=0cm]
\draw[thick,-latex] (0,0) -- (2,0);
\draw[thick,latex-] (-2,0) -- (0,0);
\draw[thick,latex-latex] (-1,-1.73205)--(1,1.73205);
\draw[thick,latex-latex] (1,-1.73205)--(-1,1.73205);
\draw[thick,-latex] (0,0)--(-3,1.73205);
\draw[thick,latex-] (3,-1.73205)--(0,0);
\draw[thick,latex-latex] (3,1.73205)--(-3,-1.73205);
\draw[thick,latex-latex] (0,3.4641)--(0,-3.4641);
\draw (2.6,-0.07) node {$\alpha_1$};
\draw (-3.5,1.9) node {$\alpha_2$};
\end{tikzpicture}
\end{center}
\end{minipage}
\caption{The root systems $\mathrm{A}_2$ (left) and $\mathrm{G}_2$ (right)
with $\Delta=\{\alpha_1,\alpha_2\}$.}
\end{figure}

Let $\R$ be an irreducible, reduced root system in a real Euclidean
space $E$ with bilinear symmetric form $(\cdot\,,\cdot)$.
Fix a base $\Delta$ of $\R$ and denote by $\R^{+}$ the set of positive roots.
Write $\alpha>0$ if $\alpha\in\R^{+}$. The Weyl vector $\rho$ is defined as
half the sum of the positive roots: $\rho=\tfrac{1}{2}\sum_{\alpha>0}\alpha$.
The height $\height(\beta)$ of the root $\beta$ is given by $\height(\beta)=
(\beta,\rho)$.
Let $r$ be the rank of $\R$ (that is, the dimension of $E$).
Then the degrees $1<d_1\leq d_2\leq\cdots\leq d_r$ of the fundamental
invariants of $\R$ are uniquely determined by
\[
\prod_{i\geq 1}\frac{1-t^{d_i}}{1-t}=\prod_{\alpha>0}
\frac{1-t^{\height(\alpha)+1}}{1-t^{\height(\alpha)}}.
\]
For example, in the standard representation of the root system
$\mathrm{A}_{n-1}$,
\begin{gather}
E=\{(x_1,\dots,x_n)\in\Real^n:~x_1+\cdots+x_n=0\}, \label{Eq_E} \\[3pt]
\Phi=\{\epsilon_i-\epsilon_j:~1\leq i\neq j\leq n\} \notag
\end{gather}
and
\[
\Delta=\{\alpha_1,\dots,\alpha_{n-1}\}=
\{\epsilon_i-\epsilon_{i+1}:~1\leq i\leq n-1\},
\]
where $\epsilon_i$ denotes the $i$th standard unit vector in $\Real^n$.
Since $\height(\epsilon_i-\epsilon_j)=j-i$,
\[
\prod_{\alpha>0}
\frac{1-t^{\height(\alpha)+1}}{1-t^{\height(\alpha)}}
=\prod_{1\leq i<j\leq n}
\frac{1-t^{j-i+1}}{1-t^{j-i}}=\prod_{i=1}^n \frac{1-t^i}{1-t}.
\]
The degrees of $\mathrm{A}_{n-1}$ are thus
$\{2,3,\dots,n\}$, and the $\mathrm{A}_{n-1}$ case of
\eqref{Eq_CT_Macdonald} is readily seen to be \eqref{Eq_CT_A}.

As a second example we consider the root system
$\mathrm{G}_2$ which is made up of two copies of $\mathrm{A}_2$---one scaled.
$E$ is \eqref{Eq_E} for $n=3$, and the canonical choice of simple roots is
given by
\[
\alpha_1=\epsilon_1-\epsilon_2 \quad\text{and}\quad
\alpha_2=2\epsilon_2-\epsilon_1-\epsilon_3.
\]
The following additional four roots complete the set of positive root $\R^{+}$:
\begin{align*}
\alpha_1+\alpha_2&=\epsilon_2-\epsilon_3, \\
2\alpha_1+\alpha_2&=\epsilon_1-\epsilon_3, \\
3\alpha_1+\alpha_2&=2\epsilon_1-\epsilon_2-\epsilon_3, \\
3\alpha_1+2\alpha_2&=\epsilon_1+\epsilon_2-2\epsilon_3.
\end{align*}
The degrees of $\mathrm{G}_2$ are now easily found to be $\{2,6\}$ and,
after the identification
$(\eup^{\epsilon_1},\eup^{\epsilon_2},\eup^{\epsilon_2})=(x,y,z)$,
the constant term identity \eqref{Eq_CT_Macdonald} becomes
\begin{align}\label{Eq_G2_equal}
\CT \bigg[& \Big(1-\frac{x^2}{yz}\Big)^k
\Big(1-\frac{y^2}{xz}\Big)^k
\Big(1-\frac{z^2}{xy}\Big)^k
\Big(1-\frac{yz}{x^2}\Big)^k
\Big(1-\frac{xz}{y^2}\Big)^k
\Big(1-\frac{xy}{z^2}\Big)^k \\
& \times \Big(1-\frac{x}{y}\Big)^k
\Big(1-\frac{x}{z}\Big)^k
\Big(1-\frac{y}{x}\Big)^k
\Big(1-\frac{y}{z}\Big)^k
\Big(1-\frac{z}{x}\Big)^k
\Big(1-\frac{z}{y}\Big)^k\bigg]  \notag \\[1mm]
&\qquad\qquad\qquad\qquad\qquad =\binom{2k}{k} \binom{6k}{k}. \notag
\end{align}
This was first proved, in independent work, by Habsieger and Zeilberger
\cite{Habsieger86,Zeilberger87}, who both utilised the $\mathrm{A}_2$ case
of Morris' constant term identity \eqref{Eq_CT_Morris}.  They in fact proved
a ($q$-analogue of a) slightly more general result related to another
conjecture of Macdonald we discuss next.

Macdonald's (ex-)conjecture \eqref{Eq_CT_Macdonald} may be generalised
by replacing the exponent $k$ on the left by $k_{\alpha}$, where
$k_{\alpha}$ depends only on the length of the root $\alpha$, i.e.,
$k_{\alpha}=k_{\beta}$ if $\|\alpha\|=\|\beta\|$, where
$\|\cdot\|:=(\cdot\,,\cdot)^{1/2}$. If $\alpha^{\vee}=2\alpha/\|\alpha\|^2$
is the coroot corresponding to $\alpha$ and
$\rho_k=\tfrac{1}{2}\sum_{\alpha>0}k_{\alpha}\alpha$,
then Macdonald's generalisation of \eqref{Eq_CT_Macdonald} is
\begin{equation}\label{Eq_CT_Macdonald_II}
\CT
\prod_{\alpha\in\R} (1-\eup^{\alpha})^{k_{\alpha}}
=\prod_{\alpha>0}
\frac{|(\rho_k,\alpha^{\vee})+k_{\alpha}|!}{|(\rho_k,\alpha^{\vee})|!}.
\end{equation}
If $k_{\alpha}$ is independent of $\alpha$, i.e., $k_{\alpha}=k$,
then $\rho_k=k\rho$ and the above right-hand side may be simplified to
that of \eqref{Eq_CT_Macdonald}.

As an example of \eqref{Eq_CT_Macdonald_II} we consider the
full Habsieger--Zeilberger theorem for $\mathrm{G}_2$
\cite{Habsieger86,Zeilberger87}.
\begin{theorem}\label{Thm_HZ}
Let $\R_s$ and $\R_l$ denote the set of short and long roots
of $\mathrm{G}_2$ respectively. Then
\begin{equation}\label{Eq_CT_G2}
\CT
\prod_{\alpha\in\R_l} (1-\eup^{\alpha})^k
\prod_{\alpha\in\R_s} (1-\eup^{\alpha})^m
=
\frac{(3k+3m)!(3k)!(2k)!(2m)!}{(3k+2m)!(2k+m)!(k+m)!k!k!m!}.
\end{equation}
\end{theorem}
Note that for $k=0$ or $m=0$ this yields \eqref{Eq_CT_A} for $n=3$.
As we shall see in Section~\ref{Sec_G2}, it is the above identity,
not it equal-parameter case \eqref{Eq_G2_equal}, that admits a logarithmic
analogue.

\section{The signatures \texorpdfstring{$\tau_{ij}$}{tau}}\label{Sec_tau}

In our discussion of complex and logarithmic constant term
identities in Sections~\ref{Sec_Complex_Morris}--\ref{Sec_Other},
an important role is played by certain signatures $\tau_{ij}$.
For the convenience of the reader, in this section we have collected
all relevant facts about the $\tau_{ij}$.

For a fixed odd positive integer $n$ and $m:=(n-1)/2$ define $\tau_{ij}$
for $1\leq i<j\leq n$ by
\begin{equation}\label{Eq_tau}
\tau_{ij}=
\begin{cases}
\ph 1 & \text{if $j\leq m+i$}, \\
-1 & \text{if $j>m+i$},
\end{cases}
\end{equation}
and extend this to all $1\leq i,j\leq n$ by setting $\tau_{ij}=-\tau_{ji}$.
Assuming that $1\leq i<n$ we have
\[
\tau_{in}=\chi(n\leq m+i)-\chi(n>m+i),
\]
where $\chi(\text{true})=1$ and $\chi(\text{false})=0$.
Since $n-m=m+1$, this is the same as
\[
\tau_{in}=\chi(i>m)-\chi(i\leq m)=-\tau_{1,i+1}=\tau_{i+1,1}.
\]
For $1\leq i,j<n$ we clearly also have $\tau_{ij}=\tau_{i+1,j+1}$.
Hence the matrix
\begin{equation}\label{Eq_Tmatrix}
\Tau:=(\tau_{ij})_{1\leq i,j\leq n}
\end{equation}
is a \emph{skew-symmetric circulant matrix}.
For example, for $n=5$,
\[
\Tau=\begin{pmatrix}[r]
0&1&1&-1&-1 \\
-1&0&1&1&-1 \\
-1&-1&0&1&1 \\
1&-1&-1&0&1 \\
1&1&-1&-1&0
\end{pmatrix}.
\]
We note that all of the row-sums (and column-sums) of the above matrix are
zero. Because $\Tau$ is a circulant matrix, to verify this property holds
for all (odd) $n$ we only needs to verify this for the first row:
\[
\sum_{j=1}^n\tau_{1j}=\sum_{j=2}^{m+1}1-\sum_{j=m+2}^n 1=m-(n-m-1)=m-m=0.
\]
By the skew symmetry, the vanishing of the row sums may also be stated as
follows.
\begin{lemma}\label{Lem_tau}
For $1\leq i\leq n$,
\[
\sum_{j=1}^{i-1}\tau_{ji}=\sum_{j=i+1}^n\tau_{ij}.
\]
\end{lemma}

A property of the signatures $\tau_{ij}$, which will be important
in our later discussions, can be neatly clarified by having recourse
to Pfaffians.

By a \emph{perfect matching} (or $1$-factor) on
$[n+1]:=\{1,2,\dots,n+1\}$ we mean a graph on the
vertex set $[n+1]$ such that each vertex has degree one, see e.g.,
\cite{Bressoud99,Stembridge90}.
If in a perfect matching $\pi$ the vertices $i<j$ are connected
by an edge we say that $(i,j)\in\pi$. Two edges $(i,j)$ and $(k,l)$
of $\pi$ are said to be crossing if $i<k<j<l$ or $k<i<l<j$.
The crossing number $c(i,j)$ of the edge $(i,j)\in\pi$ is the
number of edges crossed by $(i,j)$, and the
crossing number $c(\pi)$ is the total number
of pairs of crossing edges: $c(\pi)=\tfrac{1}{2}\sum_{(i,j)\in \pi}c(i,j)$.
We can embed perfect matching in the $xy$-plane,
such that (i) the vertex labelled $i$ occurs at the point $(i,0)$,
(ii) the edges $(i,j)$ and $(k,l)$ intersect exactly once if they are
crossing and do not intersect if they are non-crossing.
For example, the perfect matching
$\{(1,3),(2,7),(4,5),(6,8)\}$ corresponds to
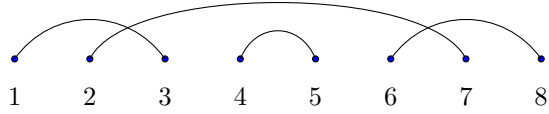
\begin{center}
\begin{tikzpicture}[scale=1]
\foreach \x in {0,...,7} \draw[fill=blue] (\x,0) circle (0.04cm);
\draw (0,-0.5) node {$1$};
\draw (1,-0.5) node {$2$};
\draw (2,-0.5) node {$3$};
\draw (3,-0.5) node {$4$};
\draw (4,-0.5) node {$5$};
\draw (5,-0.5) node {$6$};
\draw (6,-0.5) node {$7$};
\draw (7,-0.5) node {$8$};
\draw (0,0) .. controls (0.5,0.7) and (1.5,0.7) .. (2,0);
\draw (1,0) .. controls (1.5,1) and (5.5,1) .. (6,0);
\draw (3,0) .. controls (3.2,0.5) and (3.8,0.5) .. (4,0);
\draw (5,0) .. controls (5.5,0.7) and (6.5,0.7) .. (7,0);
\end{tikzpicture}
\end{center}
and has crossing number $2$ ($c(4,5)=0$, $c(1,3)=c(6,8)=1$ and $c(2,7)=2$).

The \emph{Pfaffian} of a $(2N)\times (2N)$ skew-symmetric matrix $A$
is defined as \cite{Bressoud99,Knuth96,Krattenthaler98,Stembridge90}:
\begin{equation}\label{Eq_Pfaffian}
\Pf(A):=\sum_{\pi} (-1)^{c(\pi)}\prod_{(i,j)\in\pi} A_{ij}.
\end{equation}

After these preliminaries on perfect matching and Pfaffians
we now form a second skew-symmetric matrix, closely related to $\Tau$.
First we extend the $\tau_{ij}$ to $1\leq i,j\leq n+1$
by setting $\tau_{i,n+1}=b_i$.
We then define the $(n+1)\times(n+1)$ skew-symmetric matrix
$Q(\fb,\bb)=(Q_{ij}(\fb,\bb))_{1\leq i,j\leq n+1}$, where
$\fb=(a_1,\dots,a_{n+1})$ and $\bb=(b_1,\dots,b_n)$, as follows:
\begin{equation}\label{Eq_Q-matrix}
Q_{ij}(\fb,\bb)=\tau_{ij}a_ia_j  \qquad \text{for } 1\leq i<j\leq n+1.
\end{equation}
For example, for $n=5$,
\[
Q(\fb,\bb)=\begin{pmatrix}[l]
\quad\; 0    &\ph a_1a_2   &\ph a_1a_3   &   -a_1a_4   &   -a_1a_5   &a_1a_6b_1 \\
   -a_2a_1   &\quad\; 0    &\ph a_2a_3   &\ph a_2a_4   &   -a_2a_5   &a_2a_6b_2 \\
   -a_3a_1   &   -a_3a_2   &\quad\; 0    &\ph a_3a_4   &\ph a_3a_5   &a_3a_6b_3 \\
\ph a_4a_1   &   -a_4a_2   &   -a_4a_3   &\quad\; 0    &\ph a_4a_5   &a_4a_6b_4 \\
\ph a_5a_1   &\ph a_5a_2   &   -a_5a_3   &   -a_5a_4   &\quad\; 0    &a_5a_6b_5 \\
   -a_6a_1b_1&   -a_6a_2b_2&   -a_6a_3b_3&   -a_6a_4b_4&   -a_6a_5b_5&\quad\; 0
\end{pmatrix}.
\]
Note that $\Tau$ is the submatrix of $Q\big((1^{n+1}),\bb\big)$ obtained by
deleting the last row and column.

\begin{proposition}\label{Prop_Pfaffian}
We have
\[
\Pfb\big(Q(\fb,\bb)\big)= (-1)^{\binom{m}{2}} a_1a_2\cdots a_{n+1} (b_1+b_2+\dots+b_n).
\]
\end{proposition}

\begin{proof}
The main point of our proof below is to exploit a cyclic symmetry of
the terms contributing to $\Pfb\big(Q(\fb,\bb)\big)$.
This reduces the computation of the Pfaffian to that of a sub-Pfaffian
of lower order.

Let $S(\pi;\fb,\bb)$ denote the summand of $\Pfb\big(Q(\fb,\bb)\big)$,
that is,
\[
\Pfb\big(Q(\fb,\bb)\big)=\sum_{\pi} S(\pi;\fb,\bb)\qquad\text{with}\qquad
S(\pi;\fb,\bb)=(-1)^{c(\pi)}\prod_{(i,j)\in\pi} Q_{ij}(\fb,\bb).
\]
From the definition \eqref{Eq_Q-matrix} of $Q_{ij}(\fb,\bb)$
and the fact that $\pi$ is a perfect matching on $[n+1]$,
\begin{equation}\label{Eq_S-expression}
S(\pi;\fb,\bb)=(-1)^{c(\pi)} \prod_{(i,j)\in \pi} a_ia_j\tau_{ij}=
(-1)^{c(\pi)} a_1\cdots a_{n+1}
\prod_{(i,j)\in \pi} \tau_{ij}.
\end{equation}
We now observe that $S(\pi;\fb,\bb)$ is, up to a cyclic permutation of $\bb$,
invariant under the permutation $w$ given by
$(1,2,3,\dots,n,n+1)\mapsto (n,1,2,\dots,n-1,n+1)$.
To see this, denote by $\pi'$ the image of $\pi$ under $w$.
For example, the image of the perfect matching given on the
previous page is
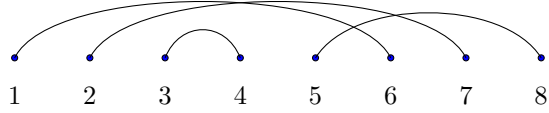
\begin{center}
\begin{tikzpicture}[scale=1]
\foreach \x in {0,...,7} \draw[fill=blue] (\x,0) circle (0.04cm);
\draw (0,-0.5) node {$1$};
\draw (1,-0.5) node {$2$};
\draw (2,-0.5) node {$3$};
\draw (3,-0.5) node {$4$};
\draw (4,-0.5) node {$5$};
\draw (5,-0.5) node {$6$};
\draw (6,-0.5) node {$7$};
\draw (7,-0.5) node {$8$};
\draw (0,0) .. controls (0.5,1) and (4.5,1) .. (5,0);
\draw (2,0) .. controls (2.2,0.5) and (2.8,0.5) .. (3,0);
\draw (4,0) .. controls (4.5,0.8) and (6.5,0.8) .. (7,0);
\draw (1,0) .. controls (1.5,1) and (5.5,1) .. (6,0);
\end{tikzpicture}
\end{center}
Under the permutation $w$, all edges not containing the vertices $1$
or $n+1$ are shifted one unit to the left: $(i,j)\mapsto (i-1,j-1)$.
For the edge $(1,j)$ containing vertex $1$ we have:
\begin{itemize}
\item[(i)]
If $j\leq n$ then $(1,j)\mapsto (j-1,n)$. This also implies that
the edge $(j',n+1)$ ($j'\geq 2$) containing vertex $n+1$ maps to
$(j'-1,n+1)$.
\item[(ii)]
If $j=n+1$ then $(1,j)=(1,n+1)\mapsto (n,n+1)=(j-1,n+1)$.
\end{itemize}
First we consider (i). If we remove the edge $(1,j)$ from $\pi$ and
carry out $w$, then the number of crossings of its image is exactly that
of $\pi$.
Hence we only need to focus on the edge $(1,j)$ and its image under $w$.
In $\pi$ the edge $(1,j)$ has crossing number $c(1,j)\equiv j\pmod{2}$,
while the edge $(j-1,n)$ in $\pi'$
has crossing number $c(j-1,n)\equiv n-j\equiv j+1\pmod{2}$.
Hence $(-1)^{c(\pi)}=-(-1)^{c(\pi')}$.
Since $\tau_{ij}=\tau_{i-1,j-1}$ (for $2\leq i<j\leq n$)
and  $\tau_{1,j}=-\tau_{j-1,n}$ it thus follows that
$\pi$ and $\pi'$ have the same sign.
Finally we note that under $w$, $b_i=\tau_{i,n+1}\mapsto
\tau_{i-1,n+1}=b_{i-1}$ (since $i\neq 1$). We thus conclude that
\begin{equation}\label{Eq_S}
S\big(\pi;\fb,(b_1,\dots,b_n)\big)\mapsto S\big(\pi';\fb,(b_2,\dots,b_n,b_1)\big),
\end{equation}
where we note that both sides depend on a single $b_i(\neq b_1)$ only.
For example, the perfect matching in the above two figures correspond to
\begin{multline*}
S\big((1,3),(2,7),(4,5),(6,8);\fb,(b_1,\dots,b_7)\big) \\
=(-1)^2 \cdot a_1a_3\cdot (-a_2a_7)\cdot a_4a_5 \cdot a_6a_8b_6=
-a_1\cdots a_8b_6
\end{multline*}
and
\begin{multline*}
S\big((1,6),(2,7),(3,4),(5,8);\fb,(b_1,\dots,b_7)\big) \\
=(-1)^3 \cdot (-a_1a_6)\cdot (-a_2a_7)\cdot a_3a_4 \times a_5a_8b_5=
-a_1\cdots a_8b_5.
\end{multline*}

The case (ii) is even simpler; the edge $(1,n+1)$ in $\pi$ and its image
$(n,n+1)$ in $\pi'$ both have crossing number $0$. The crossing numbers of
all other edges do not change by a global shift of one unit to the right,
so that $c(\pi)=c(\pi')$:

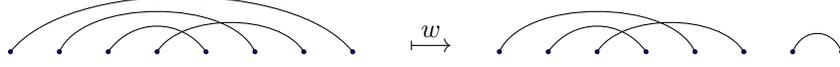
\begin{center}
\begin{tikzpicture}[scale=0.65]
\foreach \x in {0,...,7} \draw[fill=blue] (\x,0) circle (0.04cm);
\draw (0,0) .. controls (1.2,1.5) and (5.8,1.5) .. (7,0);
\draw (1,0) .. controls (1.5,1.1) and (4.5,1.1) .. (5,0);
\draw (2,0) .. controls (2.5,0.7) and (3.5,0.7) .. (4,0);
\draw (3,0) .. controls (3.5,0.8) and (5.5,0.8) .. (6,0);
\foreach \x in {10,...,17} \draw[fill=blue] (\x,0) circle (0.04cm);
\draw (16,0) .. controls (16.2,0.5) and (16.8,0.5) .. (17,0);
\draw (10,0) .. controls (10.5,1.1) and (13.5,1.1) .. (14,0);
\draw (11,0) .. controls (11.5,0.7) and (12.5,0.7) .. (13,0);
\draw (12,0) .. controls (12.5,0.8) and (14.5,0.8) .. (15,0);
\draw (8.6,0.3) node {$\stackrel{\displaystyle w}{\longmapsto}$};
\end{tikzpicture}
\end{center}

\noindent
Moreover, $\tau_{ij}=\tau_{i-1,j-1}$ (for $2\leq i<j\leq n$) so that
$\pi$ and $\pi'$ again have the same sign.
Finally, from $b_1=\tau_{1,n+1}\mapsto \tau_{n,n+1}=b_n$ it follows that
once again \eqref{Eq_S} holds,
where this time both sides depend only on $b_1$.

From \eqref{Eq_S} it follows that the Pfaffian
$\Pfb\big(Q(\fb,\bb)\big)$ is symmetric under cyclic permutations of the
$b_i$. But since the Pfaffian, viewed as a function of $\bb$, has degree $1$
it thus follows (see also \eqref{Eq_S-expression}) that
\[
\Pfb\big(Q(\fb,\bb)\big)= C a_1\cdots a_{n+1}(b_1+\cdots+b_n)
\]
for some yet-unknown constant $C$.
We shall determine $C$ by computing the coefficient of
$b_n$ of $\Pfb\big(Q((1^{n+1}),\bb\big)$,
which is equal to the Pfaffian of the
$(2m)\times(2m)$ submatrix $M$ of $\Tau$ obtained by deleting its
last row and column.

We recall the property $\Pf(M)=\Pf(U^t M U)$ of Pfaffians, where $U$
is a unipotent triangular matrix \cite{Stembridge90}.
Choosing the non-zero entries of the $(2m)\times(2m)$ matrix $U$ to be
$U_{ii}=1$ for $i=1,\dots,2m$, and $U_{i,i+m}=1$ for $i=1,\dots,m$,
one transforms $M$ into
\[
\left( \begin{array}{c|c}
M' & \eem \\[3pt] \hline
\raisebox{-3pt}{$-\eem$} & \raisebox{-3pt}{$\eez$} \end{array} \right),
\]
where $M'$ is the upper-left $m\times m$ submatrix of $M$ and
$\eem$ is the $m\times m$ identity matrix.
The Pfaffian of the above matrix, and hence that of $M$, is
exactly (cf.~\cite{Stembridge90})
$(-1)^{\binom{m}{2}} \det(\eem)=(-1)^{\binom{m}{2}}$.
This, in turn, implies that $C=(-1)^{\binom{m}{2}}$,
and the required formula follows.
\end{proof}

\begin{remark}
By a slight modification of the above proof the following more general
Pfaffian results. Let
\begin{align*}
Q_{ij}(X,\fb,\bb)&:=\tau_{ij}a_ia_j(x_i+x_j)
&&\text{for $1\leq i<j\leq n$}
\intertext{and}
Q_{i,n+1}(X,\fb,\bb)&:=\tau_{i,n+1}a_ia_{n+1}=a_ia_{n+1}b_i\hspace{-8mm}
&&\text{for $1\leq i\leq n$},
\end{align*}
and use this to form the $(n+1)\times (n+1)$ skew-symmetric matrix
$Q(X,\fb,\bb)$.
Then
\begin{multline*}
\Pfb\big(Q(X,\fb,\bb)\big) \\
=2^{m-1}(-1)^{\binom{m}{2}} a_1a_2\cdots a_{n+1}
\sum_{i=1}^n b_i (x_{i+1}\cdots x_{i+m}+x_{i+m+1}\cdots x_{i+n-1}),
\end{multline*}
where $x_{i+n}:=x_i$ for $i>n$. For $X=(1/2,\dots,1/2)$ this yields
Proposition~\ref{Prop_Pfaffian}.
\end{remark}

\section{The complex Morris constant term identity}\label{Sec_Complex_Morris}

Thanks to Lemma~\ref{Lem_tau},
\begin{align}
\label{Eq_rew}
\prod_{1\leq i\neq j\leq n} \Big(1-\frac{x_i}{x_j}\Big)&=
\prod_{1\leq i<j\leq n}
\Big(-\frac{x_j}{x_i}\Big)^{\tau_{ij}}
\Big(1-\Big(\frac{x_i}{x_j}\Big)^{\tau_{ij}}\Big)^2 \\
&=(-1)^{\binom{n}{2}}
\prod_{i=1}^n x_i^{\sum_{j=1}^{i-1}\tau_{ji}-\sum_{j=i+1}^n\tau_{ij}}
\prod_{1\leq i<j\leq n}
\Big(1-\Big(\frac{x_i}{x_j}\Big)^{\tau_{ij}}\Big)^2 \notag \\
&=(-1)^m
\prod_{1\leq i<j\leq n}
\Big(1-\Big(\frac{x_i}{x_j}\Big)^{\tau_{ij}}\Big)^2.
\notag
\end{align}
For odd values of $n$ Morris' constant term identity \eqref{Eq_CT_Morris}
can thus be rewritten in the equivalent form
\begin{multline}\label{Eq_CT_Morris_tau}
\CT\bigg[\,\prod_{i=1}^n
(1-x_i)^a\Big(1-\frac{1}{x_i}\Big)^b
\prod_{1\leq i<j\leq n} \Big(1-\Big(\frac{x_i}{x_j}\Big)^{\tau_{ij}}\Big)^{2k}
\bigg] \\
=(-1)^{km}\prod_{i=0}^{n-1} \frac{(a+b+ik)!((i+1)k)!}{(a+ik)! (b+ik)! k!}.
\end{multline}
The crucial point about this rewriting is that in the product
\[
\prod_{1\leq i<j\leq n} \Big(1-\Big(\frac{x_i}{x_j}\Big)^{\tau_{ij}}\Big)^{2k}
\]
each of the variables
$x_1,x_2,\dots,x_n$ occurs exactly $m$ times in one of the numerators
and $m$ times in one of the denominators.
For example,
\[
\prod_{1\leq i<j\leq 3} \Big(1-\Big(\frac{x_i}{x_j}\Big)^{\tau_{ij}}\Big)^{2k}
=\Big(1-\frac{x_1}{x_2}\Big)^{2k}
\Big(1-\frac{x_2}{x_3}\Big)^{2k}
\Big(1-\frac{x_3}{x_1}\Big)^{2k}.
\]
Obviously, for $n$ even such a rewriting is not possible.

We are now interested in the question as to what happens when $2k$ is replaced
by an arbitrary complex variable $u$.
For $n=3$ we will later prove the following proposition.
\begin{proposition}\label{Prop_CT_A2_complex}
For $a,b$ nonnegative integers and $\Re(1+\tfrac{3}{2}u)>0$,
\begin{multline}\label{Eq_n3}
\CT\bigg[
(1-x)^a(1-y)^a(1-z)^a
\Big(1-\frac{1}{x}\Big)^b \Big(1-\frac{1}{y}\Big)^b \Big(1-\frac{1}{z}\Big)^b \\
\qquad\qquad\qquad\qquad\qquad\qquad\times
\Big(1-\frac{x}{y}\Big)^u
\Big(1-\frac{y}{z}\Big)^u
\Big(1-\frac{z}{x}\Big)^u \bigg] \\
=\cos\big(\tfrac{1}{2}\pi u\big)
\frac{\Gamma(1+\frac{3}{2}u)}{\Gamma^3(1+\frac{1}{2}u)}
\prod_{i=0}^2 \frac{(1+\frac{1}{2}iu)_{a+b}}
{(1+\frac{1}{2}iu)_a(1+\frac{1}{2}iu)_b}.
\end{multline}
\end{proposition}
As follows from its proof, a slightly more general result in fact holds.
Using $(z)_{n+m}=(z)_n(z+n)_m$ and $(1-x)^a(1-x^{-1})^b=
(-x)^{-b}(1-x)^{a+b}$, then replacing $a\mapsto a-b$, and
finally using $(z-b)_b=(-1)^b (1-z)_b$, the identity \eqref{Eq_n3}
can also be stated as
\begin{multline}\label{Eq_n3_b}
\big[x^by^bz^b\big] \bigg[
(1-x)^a(1-y)^a(1-z)^a
\Big(1-\frac{x}{y}\Big)^u
\Big(1-\frac{y}{z}\Big)^u
\Big(1-\frac{z}{x}\Big)^u \bigg] \\
=\cos\big(\tfrac{1}{2}\pi u\big)
\frac{\Gamma(1+\frac{3}{2}u)}{\Gamma^3(1+\frac{1}{2}u)}
\prod_{i=0}^2 \frac{(-a-\frac{1}{2}iu)_b}{(1+\frac{1}{2}iu)_b},
\end{multline}
where $\big[X^c]f(X)$ (with $X^c=x_1^{c_1}\cdots x_n^{c_n}$) denotes the
coefficient of $X^c$ in $f(X)$. This alternative form
of \eqref{Eq_n3} is true for all $a,u\in\Complex$ such that
$\Re(1+\tfrac{3}{2}u)>0$.

In view of Proposition~\ref{Prop_CT_A2_complex} it seems reasonable to
make the following more general conjecture.

\begin{conjecture}[\textbf{Complex Morris constant term identity}]
\label{Conj_CT_An_complex}
Let $n$ be an odd positive integer, $a,b$ nonnegative integers
and $u\in\Complex$ such that $\Re(1+\tfrac{1}{2}nu)>0$.
Then there exists a polynomial $P_n(x)$,
independent of $a$ and $b$, such that
$P_n(0)=1/(n-2)!!$, $P_n(1)=1$, and
\begin{multline}\label{Eq_complex_Morris}
\CT\bigg[\,\prod_{i=1}^n
(1-x_i)^a\Big(1-\frac{1}{x_i}\Big)^b
\prod_{1\leq i<j\leq n} \Big(1-\Big(\frac{x_i}{x_j}\Big)^{\tau_{ij}}\Big)^u
\bigg] \\
=x^m P_n(x^2)
\frac{\Gamma(1+\frac{1}{2}nu)}{\Gamma^n(1+\frac{1}{2}u)}
\prod_{i=0}^{n-1} \frac{(1+\frac{1}{2}iu)_{a+b}}
{(1+\frac{1}{2}iu)_a(1+\frac{1}{2}iu)_b},
\end{multline}
where $x=x(u):=\cos\big(\frac{1}{2}\pi u\big)$ and $m:=(n-1)/2$.
\end{conjecture}

Note that for $u$ an odd positive integer the kernel on the left of
\eqref{Eq_complex_Morris} is a skew-symmetric function, so that its
constant term trivially vanishes. When $u$ is an even integer, say $2k$ then
$x=\cos(\pi k)=(-1)^k$ so that $x^m P_n(x^2)=(-1)^{km} P_n(1)=(-1)^{km}$
in accordance with \eqref{Eq_CT_Morris_tau}.
Similar to the case $n=3$, in the form
\begin{multline*}
\big[(x_1\cdots x_n)^b \big] \bigg[
\CT\bigg[\,\prod_{i=1}^n (1-x_i)^a
\prod_{1\leq i<j\leq n} \Big(1-\Big(\frac{x_i}{x_j}\Big)^{\tau_{ij}}\Big)^u
\bigg] \\
=x^m P_n(x^2)
\frac{\Gamma(1+\frac{1}{2}nu)}{\Gamma^n(1+\frac{1}{2}u)}
\prod_{i=0}^{n-1} \frac{(-a-\frac{1}{2}iu)_b}
{(1+\frac{1}{2}iu)_b}
\end{multline*}
Conjecture \ref{Conj_CT_An_complex} should hold for all $a\in\Complex$.

For $n=1$ the left-side of \eqref{Eq_complex_Morris} does not depend on $u$
so that $P_1(x)=1$. Moreover, from Proposition~\ref{Prop_CT_A2_complex} it
follows that also $P_3(x)=1$.
Extensive numerical computations leave little doubt that the next two
instances of $P_n(x)$ are given by
\begin{align*}
P_5(x)&=\frac{1}{3}(1+2x) \\
P_7(x)&=\frac{1}{45}(3+26x-16x^2+32x^3).
\end{align*}
Conjecturally, we also have $\deg(P_n(x))=\binom{m}{2}$ and
\begin{gather*}
P_n'(0)=2\binom{m}{2}\frac{2n-1}{9(n-2)!!} \\[2mm]
P_n'(1)=\frac{2}{3}\binom{m}{2},\qquad
P''_n(1)=\frac{2}{45}\binom{m}{3}(19m+23),
\end{gather*}
but beyond this we know very little about $P_n(x)$.

To conclude our discussion of the polynomials $P_n(x)$ we note that
if $z_i=z_i(u):=\cos(i \pi u)$, then
\begin{align*}
P_5\big(x^2(u)\big)&=\frac{1}{3}(2+z_1) \\
P_7\big(x^2(u)\big)&=\frac{1}{45}(20+20z_1+4z_2+z_3),
\end{align*}
suggesting that the coefficients of $z_i$ admit a combinatorial
interpretation.

As will be shown in the next section, the complex Morris constant term
identity \eqref{Eq_complex_Morris} implies the logarithmic
Morris constant term identity \eqref{Eq_log_CT_Morris}, and the only
properties of $P_n(x)$ that are essential in the proof are
$P_n(0)=1/(n-2)!!$ and $P_n(1)=1$.

To conclude this section we give a proof of
Proposition~\ref{Prop_CT_A2_complex}.
The reader unfamiliar with the basic setup of the method of creative
telescoping is advised to consult the text~\cite{PWZ96}.

\begin{proof}[Proof of Proposition \ref{Prop_CT_A2_complex}]
Instead of proving \eqref{Eq_n3} we establish the slightly more general
\eqref{Eq_n3_b}.

By a six-fold use of the binomial expansion \eqref{Eq_binomial_thm}, the
constant term identity \eqref{Eq_n3_b} can be written as the following
combinatorial sum:
\begin{multline*}
\sum_{m_0,m_1,m_2=0}^{\infty} \prod_{i=0}^2 (-1)^{m_i}
\binom{u}{m_i} \binom{a}{b+m_i-m_{i+1}} \\
=\cos\big(\tfrac{1}{2}\pi u\big)\,
\frac{\Gamma(1+\frac{3}{2}u)}{\Gamma^3(1+\frac{1}{2}u)}
\prod_{i=0}^2 \frac{(-a-\frac{1}{2}iu)_b}{(1+\frac{1}{2}iu)_b},
\end{multline*}
where $m_3:=m_0$ and where $a,u\in\Complex$ such that
$\Re(1+\frac{3}{2}u)>0$ and $b$ is a nonnegative integer.
If we denote the summand of this identity by $f_b(\frac{1}{2}u,-1-a;\bm)$
where $\bm:=(m_0,m_1,m_2)$, then we need to prove that
\begin{equation}\label{Eq_F}
F_b(u,v):=\sum_{\bm\in\Z^3}f_b(u,v;\bm)
=\cos\big(\pi u\big)\,
\frac{\Gamma(1+3u)}{\Gamma^3(1+u)}
\prod_{i=0}^2 \frac{(1+v-iu)_b}{(1+iu)_b},
\end{equation}
for $\Re(1+3u)>0$.

In our working below we suppress the dependence of the
various functions on the variables $u$ and $v$.
In particular we write $F_b$ and $f_b(\bm)$
for $F_b(u,v)$ and $f_b(u,v;\bm)$.

The function $f_0(\bm)$ vanishes unless $m_0=m_1=m_2$.
Hence
\[
F_0=\sum_{m=0}^{\infty} (-1)^m \binom{2u}{m}^3=
\hyper{3}{2}{-2u,-2u,-2u}{1,1}{1},
\]
where we adopt standard notation for (generalised) hypergeometric series,
see e.g., \cite{AAR99,Bailey35}.
The $_3F_2$ series is summable by the $2a=b=c=-2u$ case of Dixon's sum
\cite[Eq. (2.2.11)]{AAR99}
\begin{multline}\label{Eq_Dixon}
\hyper{3}{2}{2a,b,c}{1+2a-b,1+2a-c}{1} \\
=\frac{\Gamma(1+a)\Gamma(1+2a-b)\Gamma(1+2a-c)
\Gamma(1+a-b-c)}
{\Gamma(1+2a)\Gamma(1+a-b)\Gamma(1+a-c)\Gamma(1+2a-b-c)}
\end{multline}
for $\Re(1+a-b-c)>0$.
As a result,
\[
F_0=
\frac{\Gamma(1-u)\Gamma(1+u)}{\Gamma(1-2u)\Gamma(1+2u)}\cdot
\frac{\Gamma(1+3u)}{\Gamma^3(1+u)}=
\cos(\pi u)\,
\frac{\Gamma(1+3u)}{\Gamma^3(1+u)},
\]
proving the $b=0$ instance of \eqref{Eq_F}.

In the remainder we assume that $b\geq 1$.

Let $\mathcal{C}$ be the generator of the cyclic group $C_3$ acting
on $\bm$ as $\mathcal{C}(\bm)=(m_2,m_0,m_1)$.
With the help of the multivariable Zeilberger algorithm \cite{AZ06},
one discovers the (humanly verifiable) rational function identity
\begin{multline}\label{Eq_Zeilberger}
t_b(\bm)\prod_{i=0}^2 (b+i u) + \prod_{i=0}^2 (b+v-iu) \\
=\sum_{i=0}^2
\Big(r_b\big(\boldsymbol{e}_1+\mathcal{C}^i(\bm)\big)\,
s_b\big(\mathcal{C}^i(\bm)\big)
+r_b\big(\mathcal{C}^i(\bm)\big)\Big),
\end{multline}
where
\begin{align*}
r_b(\bm)&=-\frac{m_0(b+v+m_2-m_0)}{6(b+m_1-m_2)(b+m_2-m_0)}\\
&\qquad \times \bigl((2b+v)(3b^2+3bv+2uv)
+2(m_1-m_2)(3b^2+3bv+v^2-uv)\bigr),
\\
s_b(\bm)&=-\frac{f_{b-1}(\boldsymbol{e}_1+\bm)}{f_{b-1}(\bm)}
=\frac{(2u-m_0)(b+v+m_0-m_1)(b+m_2-m_0-1)}
{(1+m_0)(b+m_0-m_1)(b+v+m_2-m_0-1)},
\\
t_b(\bm)&=-\frac{f_b(\bm)}{f_{b-1}(\bm)}=
\prod_{i=0}^2 \frac{b+v+m_i-m_{i+1}}{b+m_i-m_{i+1}},
\end{align*}
and $\eb_1+\bm:=(1+m_0,m_1,m_2)$.
If we multiply \eqref{Eq_Zeilberger} by $-f_{b-1}(\bm)$ and
use that
$f_b(\bm)=f_b(\mathcal{C}^i(\bm))$ we find that
\begin{multline*}
f_b(\bm)\prod_{i=0}^2 (b+i u)-f_{b-1}(\bm) \prod_{i=0}^2 (b+v-iu) \\
=\sum_{i=0}^2
\Big[r_b\big(\eb_1+\mathcal{C}^i(\bm)\big)\,
f_{b-1}\big(\eb_1+\mathcal{C}^i(\bm)\big)
-r_b\big(\mathcal{C}^i(\bm)\big)f_{b-1}\big(\mathcal{C}^i(\bm)\big)\Big],
\end{multline*}
Summing this over $\bm\in\mathbb Z^3$ the right-hand
side telescopes to zero, resulting in
\[
F_b=F_{b-1} \prod_{i=0}^2 \frac{(b+v-iu)}{(b+iu)}.
\]
By $b$-fold iteration this yields
\[
F_b=F_0\prod_{i=0}^2 \frac{(1+v-iu)_b}{(1+iu)_b}.  \qedhere
\]
\end{proof}

\section{The logarithmic Morris constant term identity}\label{Sec_CT_Log_Morris}

This section contains three parts. In the first very short part,
we present an integral analogue of the logarithmic Morris constant
term identity. This integral may be viewed as a logarithmic version of
the well-known Morris integral.
The second and third, more substantial parts, contain respectively a
proof of Theorem~\ref{Thm_log_CT_Morris} and, exploiting
some further results of Adamovi\'c and Milas, a strengthening
of this theorem.

\subsection{A logarithmic Morris integral}
By a repeated use of Cauchy's integral formula, constant term
identities such as \eqref{Eq_CT_Morris} or \eqref{Eq_log_CT_Morris}
may be recast in the form of multiple integral evaluations.
In the case of \eqref{Eq_CT_Morris} this
leads to the well-known Morris integral \cite{FW08,Morris82}
\begin{multline*}
\Int_{[-\frac{1}{2}\pi,\frac{1}{2}\pi]^n}
\prod_{i=1}^n \eup^{\iup(a-b)\theta_i}
\sin^{a+b}(\theta_i)
\prod_{1\leq i<j\leq n}
\sin^{2k}(\theta_i-\theta_j)
\dup\theta_1\cdots\dup\theta_n \\
=\big(B_{k,n}(a,b)\big)^n
\prod_{i=0}^{n-1} \frac{(a+b+ik)!((i+1)k)!}{(a+ik)! (b+ik)! k!},
\end{multline*}
where $B_{k,n}(a,b)=\pi \iup^{a-b} 2^{-k(n-1)-a-b}$. The Morris integral
may be shown to be a simple consequence of the Selberg integral
\cite{FW08,Selberg44}.
Thanks to \eqref{Eq_log_CT_Morris} we now have a logarithmic
analogue of the Morris integral as follows:
\begin{multline*}
\Int_{[-\frac{1}{2}\pi,\frac{1}{2}\pi]^n}
\prod_{i=1}^n \eup^{\iup(a-b)\theta_i}
\sin^{a+b}(\theta_i)
\prod_{i=1}^m \log\big(1-\eup^{2\iup(\theta_{2i}-\theta_{2i-1})}\big) \\[-2mm]
\qquad\qquad\qquad\qquad \times
\prod_{1\leq i<j\leq n}
\sin^K(\theta_i-\theta_j)
\dup\theta_1\cdots\dup\theta_n \\
=\big(C_{k,n}(a,b)\big)^n \,\frac{1}{n!!}
\prod_{i=0}^{n-1}
\frac{(2a+2b+iK)!!((i+1)K)!!}{(2a+iK)!!(2b+iK)!!K!!},
\end{multline*}
where $C_{k,n}(a,b)=\pi \iup^{a-b-m} 2^{-Km-a-b}$. Unfortunately, this cannot
be rewritten further in a form that one could truly call a logarithmic
Selberg integral.

\subsection{Proof of Theorem~\ref{Thm_log_CT_Morris}}
In this subsection we prove that the logarithmic Morris constant term
identity \eqref{Eq_log_CT_Morris} is nothing but the $m$th derivative of the
complex Morris constant term identity \eqref{Eq_complex_Morris}
evaluated at $u=K:=2k+1$.

To set things up we first prepare a technical lemma.
For $\Symm_n$ the symmetric group on $n$ letters and $w\in\Symm_n$,
we denote by $\sgn(w)$ the signature of the permutation $w$, see
e.g., \cite{Macdonald95}.
The identity permutation in $\Symm_n$ will be written as $\een$.

\begin{lemma}\label{Lem_CL}
For $n$ an odd integer, set $m:=(n-1)/2$.
Let $t_{ij}$ for $1\leq i<j\leq n+1$ be a collection of signatures
\textup(i.e., each $t_{ij}$ is either $+1$ or $-1$\textup)
such that $t_{i,n+1}=1$, and $\tilde{Q}$ a skew-symmetric
matrix with entries $\tilde{Q}_{ij}=t_{ij}$ for $1\leq i<j\leq n+1$.

If  $f(X)$ is a skew-symmetric polynomial in $X=(x_1,\dots,x_n)$,
$g(z)$ a Laurent polynomial or Laurent series in the scalar variable $z$,
and $g_{ij}(X):=g((x_i/x_j)^{t_{ij}})$, then the following statements hold.

\textup{(1)}
For $w\in\Symm_{n}$, denote
$g(w;X):=\prod_{k=1}^m g(x_{w_{2k-1}}/ x_{w_{2k}})$.  Then
\[
\CT\big[ f(X) g(w;X) \big] = \sgn(w) \CT\big[ f(X) g(\een;X) \big].
\]

\textup{(2)}
For $\pi$ a perfect matching on $[n+1]$,
\begin{equation} \label{perfect_sum}
\sum_{\pi}\CT\bigg[f(X) \prod_{\substack{(i,j)\in\pi \\ j\neq n+1}}g_{ij}(X) \bigg]
=\Pf(\tilde{Q}) \CT\big[ f(X) g(\een;X) \big] .
\end{equation}
\end{lemma}

We will be needing a special case of this lemma corresponding to
$t_{ij}=\tau_{ij}$ for $1\leq i,j\leq n$, with the $\tau_{ij}$
defined in~\eqref{Eq_tau}.  Then the matrix $\tilde{Q}$ is coincides
with $Q\big((1^{n+1}),(1^n)\big)$
of \eqref{Eq_Q-matrix}, so that by Lemma~\ref{Prop_Pfaffian},
$\Pf(\tilde{Q})=(-1)^{\binom{m}{2}}n$. We summarise this in the
following corollary.

\begin{corollary}\label{Cor_CT}
If in Lemma~\ref{Lem_CL} we specialise
$t_{ij}=\tau_{ij}$ for $1\leq i<j\leq n$, then
\begin{equation}\label{Eq_PM}
\sum_{\pi}
\CT\bigg[f(X) \prod_{\substack{(i,j)\in\pi \\ j\neq n+1}}
g_{ij}(X)\bigg]=(-1)^{\binom{m}{2}} n
\CT\big[f(X)  g(\een;X) \big].
\end{equation}
\end{corollary}

\begin{proof}[Proof of Lemma~\ref{Lem_CL}]
(1) According to the ``Stanton--Stembridge trick''
\cite{Stanton86,Stembridge88,Zeilberger90},
\[
\CT\big[h(X)\big]=\CT\big[w\big(h(X)\big)\big]
\quad\text{for $w\in\Symm_n$},
\]
where $w(h(X))$ is shorthand for $h(x_{w_1},\dots,x_{w_n})$.

For our particular choice of $h$, the skew-symmetric
factor $f(X)$ produces the claimed sign.

(2) A permutation $w\in\Symm_{n}$ may be interpreted as a signed perfect matching
$(-1)^{d(w)}(w_{1},w_{2})\cdots (w_{n-2},w_{n-1}) (w_{n},w_{n+1})$, where $d(w)$ counts
the number $\abs{\{ k\leq m:\, w_{2k-1}>w_{2k}\}}$.
By claim~(1), the left hand-side of~\eqref{perfect_sum} is a multiple of
$ \CT\big[ f(X) g(\een;X) \big] $; the factor is exactly the sum $\sum_{\pi}  (-1)^{c(\pi)} \prod t_{ij}$,
in which one recognises the Pfaffian of~$\tilde Q$.
\end{proof}

\begin{proof}[Conditional proof of \eqref{Eq_log_CT_Morris}]
Suppressing the $a$ and $b$ dependence, denote the left and right-hand sides
of \eqref{Eq_complex_Morris} by $L_n(u)$ and $R_n(u)$ respectively.
We then wish to show that \eqref{Eq_log_CT_Morris} is identical to
\[
L_n^{(m)}(K)=R_n^{(m)}(K).
\]

Let us first consider the right hand side, which we write as
$R_n(u)=p_n(u)r_n(u)$, where
\[
p_n(u)=x^m P_n(x^2), \qquad x=x(u)=\cos\big(\tfrac{1}{2}\pi u\big)
\]
and
\begin{equation}\label{Eq_rnu}
r_n(u)=\frac{\Gamma(1+\frac{1}{2}nu)}{\Gamma^n(1+\frac{1}{2}u)}
\prod_{i=0}^{n-1} \frac{(1+\frac{1}{2}iu)_{a+b}}
{(1+\frac{1}{2}iu)_a(1+\frac{1}{2}iu)_b}.
\end{equation}
Since $x(K)=0$, it follows that for $0\leq j\leq m$,
\begin{equation}\label{Eq_pn}
p_n^{(j)}(K)=
(-1)^{km+m} \Big(\frac{\pi}{2}\Big)^m \, \frac{m!}{(n-2)!!} \, \delta_{jm}.
\end{equation}
Therefore, since $r_n(u)$ is $m$ times differentiable at $u=K$,
\begin{equation}\label{Eq_Rnm}
R_n^{(m)}(K)=p_n^{(m)}(K)r_n(K).
\end{equation}
By the functional equation for the gamma function
\begin{equation}\label{Eq_Legendre}
\Gamma(1+\tfrac{1}{2}N)=\begin{cases}
N!!\,2^{-N/2} \sqrt{\pi/2} &\text{if $N>0$ is odd}, \\[2mm]
N!!\,2^{-N/2} &\text{if $N\ge0$ is even},
\end{cases}
\end{equation}
and, consequently,
\begin{equation}\label{Eq_phfrac}
(1+\tfrac{1}{2}N)_a=\frac{(N+2a)!!}{2^a N!!}
\end{equation}
for any nonnegative integer $N$. Applying these formulae to
\eqref{Eq_rnu} with $u=K$, we find that
\begin{equation}\label{rnK}
r_n(K)=\Big(\frac{2}{\pi}\Big)^m
\prod_{i=0}^{n-1} \frac{(2a+2b+iK)!!((i+1)K)!!}
{(2a+iK)!!(2b+iK)!!K!!}.
\end{equation}
Combined with \eqref{Eq_pn} and \eqref{Eq_Rnm} this implies
\begin{equation}\label{Eq_Rfinal}
R_n^{(m)}(K)=(-1)^{(k+1)m} \, \frac{m!}{(n-2)!!}
\prod_{i=0}^{n-1} \frac{(2a+2b+iK)!!((i+1)K)!!}
{(2a+iK)!!(2b+iK)!!K!!}.
\end{equation}

\medskip

Next we focus on the calculation of $L_n^{(m)}(K)$.
To keep all equations in check we define
\[
f_{ab}(X):=\prod_{i=1}^n (1-x_i)^a\Big(1-\frac{1}{x_i}\Big)^b.
\]
and
\begin{equation}\label{Eq_Fab}
F_{ab}(X):=
\Delta(X)\prod_{i=1}^n x_i^{-m} (1-x_i)^a\Big(1-\frac{1}{x_i}\Big)^b
\prod_{1\leq i\neq j\leq n}\Big(1-\frac{x_i}{x_j}\Big)^k.
\end{equation}

Let $\ib:=(i_1,\dots,i_m)$ and $\jb:=(j_1,\dots,j_m)$.
Then, by a straightforward application of the product rule,
\begin{equation}\label{gnmu}
L_n^{(m)}(u)=
\sum_{1\leq i_1<j_1\leq n} \cdots
\sum_{1\leq i_m<j_m\leq n}
\L_{n;\ib,\jb}(u),
\end{equation}
where
\[
\L_{n;\ib,\jb}(u)=
\CT\bigg[f_{ab}(X)
\prod_{1\leq i<j\leq n}\Big(1-\Big(\frac{x_i}{x_j}\Big)^{\tau_{ij}}\Big)^u
\prod_{\ell=1}^m \log\Big(1-\Big(\frac{x_{i_{\ell}}}{x_{j_{\ell}}}
\Big)^{\tau_{i_{\ell}j_{\ell}}}\Big)
\bigg].
\]
For $u=K$ the kernel without the product over logarithms is a skew-symmetric
function in $X$, so that $\L_{n;\ib,\jb}(K)=0$ if there exists
a pair of variables, say $x_r$ and $x_s$, that does not occur in the
product of logarithms.
In other words, $\L_{n;\ib,\jb}(K)$ vanishes unless all of the $2m=n-1$
entries of $\ib$ and $\jb$ are distinct:
\[
L_n^{(m)}(K)=
\sum
\CT\bigg[f_{ab}(X)
\prod_{1\leq i<j\leq n}\Big(1-\Big(\frac{x_i}{x_j}\Big)^{\tau_{ij}}\Big)^K
\prod_{\ell=1}^m \log\Big(1-\Big(\frac{x_{i_{\ell}}}{x_{j_{\ell}}}
\Big)^{\tau_{i_{\ell}j_{\ell}}}\Big)\bigg],
\]
where the sum is over $1\leq i_{\ell}<j_{\ell}\leq n$ for $1\leq \ell\leq m$
such that all of $i_1,\dots,i_m,j_1,\dots,j_m$ are distinct.
By \eqref{Eq_rew} and
\[
\prod_{1\leq i<j\leq n} \Big(1-\Big(\frac{x_i}{x_j}\Big)^{\tau_{ij}}\Big)
=(-1)^{\binom{m}{2}}\Delta(X) \prod_{i=1}^n x_i^{-m}
\]
this can be simplified to
\[
L_n^{(m)}(K)=(-1)^{km+\binom{m}{2}} \sum \CT\bigg[
F_{ab}(X) \prod_{\ell=1}^m \log\Big(1-\Big(\frac{x_{i_{\ell}}}{x_{j_{\ell}}}
\Big)^{\tau_{i_{\ell}j_{\ell}}}\Big)\bigg].
\]
Using the $\Symm_m$ symmetry of the product over the logarithmic terms,
this reduces further to
\[
L_n^{(m)}(K)=(-1)^{km+\binom{m}{2}} m!
\sum \CT\bigg[F_{ab}(X)
\prod_{\ell=1}^m \log\Big(1-\Big(\frac{x_{i_{\ell}}}{x_{j_{\ell}}}
\Big)^{\tau_{i_{\ell}j_{\ell}}}\Big)\bigg],
\]
where $1\leq i_{\ell}<j_{\ell}\leq n$ for
$1\leq \ell\leq m$ such that $i_1<i_2<\cdots<i_m$, and all of
$i_1,\dots,i_m,j_1,\dots,j_m$ are pairwise distinct.

For the term in the summand corresponding to $\ib,\jb$ there is exactly
one integer $\ell$ in $[n]$ such that $\ell\not\in\ib,\jb$. Pair this
integer with $n+1$ to form the edge $(\ell,n+1)$ in a perfect matching
on $[n+1]$.
The other edges of this perfect matching are given by the
$m$ distinct pairs $(i_1,j_1),\dots,(i_m,j_m)$.
Hence
\[
L_n^{(m)}(K)=(-1)^{km+\binom{m}{2}} m!
\sum_{\pi}  \CT\bigg[F_{ab}(X)
\prod_{\substack{(i,j)\in\pi \\ j\neq n+1}} \log\Big(1-\Big(\frac{x_i}{x_j}
\Big)^{\tau_{ij}}\Big)\bigg].
\]
Since $F_{ab}(X)$ is a skew-symmetric function (it is the product of
a symmetric function times the skew-symmetric Vandermonde product $\Delta(X))$
we are in a position to apply Corollary~\ref{Cor_CT}. Thus
\[
L_n^{(m)}(K)=
(-1)^{km} n\, m! \CT\bigg[F_{ab}(X)
\prod_{i=1}^m \log\Big(1-\frac{x_{{2i-1}}}{x_{2i}}\Big) \bigg].
\]
Finally we replace $X\mapsto X^{-1}$ using
$F_{ab}(X^{-1})=(-1)^m F_{ba}(X)$, and use the symmetry in $a$ and $b$
to find
\[
L_n^{(m)}(K)=(-1)^{(k+1)m}
n \, m! \CT\bigg[F_{ab}(X)
\prod_{i=1}^m \log\Big(1-\frac{x_{{2i}}}{x_{2i-1}}\Big) \bigg].
\]
Equating this with \eqref{Eq_Rfinal} completes the proof
of \eqref{Eq_log_CT_Morris}.
\end{proof}

\subsection{A strengthening of Theorem~\ref{Thm_log_CT_Morris}}\label{Sec_New}
As will be described in more detail below,
using some further results of Adamovi\'c and Milas, it
follows that the logarithmic Morris constant term identity
\eqref{Eq_log_CT_Morris} holds provided it holds for $a=b=0$, i.e.,
provided the logarithmic analogue \eqref{Eq_CT_A} of Dyson's identity holds.
The proof of Theorem~\ref{Thm_log_CT_Morris} given in the previous
subsection implies that the latter follows from what could be termed
the complex analogue of Dyson's identity, i.e., the $a=b=0$ case
of \eqref{Eq_complex_Morris}:
\begin{equation}\label{Eq_complex_Dyson}
\CT\bigg[\,\prod_{i=1}^n
\prod_{1\leq i<j\leq n} \Big(1-\Big(\frac{x_i}{x_j}\Big)^{\tau_{ij}}\Big)^u
\bigg]
=x^m P_n(x^2)
\frac{\Gamma(1+\frac{1}{2}nu)}{\Gamma^n(1+\frac{1}{2}u)}.
\end{equation}
As a consequence of all this, Theorem~\ref{Thm_log_CT_Morris} can be
strengthened as follows.
\begin{theorem}[\textbf{Logarithmic Morris constant term identity,
strong version}]
The complex Dyson constant term identity \eqref{Eq_complex_Dyson}
implies the logarithmic Morris constant term identity.
\end{theorem}

To justify this claim, let $e_r(X)$ for $r=0,1,\dots,n$
denote the $r$th elementary symmetric function.
The $e_r(X)$ have generating function \cite{Macdonald95}
\begin{equation}\label{Eq_e_gen}
\sum_{r=0}^n z^r e_r(X)=\prod_{i=1}^n (1+zx_i).
\end{equation}
Recalling definition \eqref{Eq_Fab} of $F_{ab}$, we now define
$f_r(a)=f_r(a,b,k,n)$ by
\[
f_r(a)=\CT\big[(-1)^r e_r(X) G_{ab}(X)\big],
\]
where
\[
G_{ab}(X)=F_{ab}(X)
\prod_{i=1}^m \log\Big(1-\frac{x_{{2i}}}{x_{2i-1}}\Big).
\]
In the following $b$ may be viewed as a formal or complex variable,
but $a$ must be taken to be an integer.

From \eqref{Eq_e_gen} with $z=-1$ it follows that
\begin{equation}\label{Eq_fr_gen}
\sum_{r=0}^n f_r(a)=\CT\big[G_{a+1,b}(X)\big]=f_0(a+1).
\end{equation}
According to \cite[Theorem 7.1]{AM11} (translated into the notation
of this paper) we also have
\begin{equation}\label{Eq_rec_AM}
(n-r)(2b+rK)f_r(a)=(r+1)(2a+2+(n-r-1)K)f_{r+1}(a),
\end{equation}
where we recall that $K:=2k+1$.
Iterating this recursion yields
\[
f_r(a)=f_0(a) \binom{n}{r}
\prod_{i=0}^{r-1} \frac{2b+iK}{2a+2+(n-i-1)K}.
\]
Summing both sides over $r$ and using \eqref{Eq_fr_gen} leads to
\[
f_0(a+1)=f_0(a) \, \hyper{2}{1}{-n,2b/K}{1-n-(2a+2)/K}{1}.
\]
The $_2F_1$ series sums to $((2a+2b+2)/K)_n/((2a+2)K)_n$
by the Chu--Vandermonde sum \cite[Corollary 2.2.3]{AAR99}.
Therefore,
\[
f_0(a+1)=f_0(a) \prod_{i=0}^{n-1} \frac{2a+2b+2+iK}{2a+2+iK}.
\]
This functional equation can be solved to finally yield
\[
f_0(a)=f_0(0)
\prod_{i=0}^{n-1}
\frac{(2a+2b+iK)!!(iK)!!}{(2b+iK)!!(2a+iK)!!}.
\]

To summarise the above computations, we have established that
\[
\CT\big[G_{ab}(X)\big]=
\CT\big[G_{0,b}(X)\big]
\prod_{i=0}^{n-1}
\frac{(2a+2b+iK)!!(iK)!!}{(2b+iK)!!(2a+iK)!!}.
\]
But since $G_{0,0}(X)$ is homogeneous of degree $0$ it follows that
\[
\CT\big[G_{0,b}(X)\big]=\CT\big[G_{0,0}(X)\big],
\]
so that indeed the logarithmic Morris constant term identity
is implied by its $a=b=0$ case.

We finally remark that it seems highly plausible that the
recurrence \eqref{Eq_rec_AM} has the following analogue for
the complex Morris identity (enhanced by the term $(-1)^r e_r(X)$
in the kernel):
\[
(n-r)(2b+ru)f_r(a)=(r+1)(2a+2+(n-r-1)u)f_{r+1}(a).
\]
However, the fact that for general complex $u$ the kernel
is not a skew-symmetric function seems to prevent the proof of
\cite[Theorem 7.1]{AM11} carrying over to the complex case
in a straightforward manner.

\section{The root system \texorpdfstring{$\mathrm{G}_2$}{G\_2}}\label{Sec_G2}

In this section we prove complex and logarithmic analogues of the
Habsieger--Zeilberger identity \eqref{Eq_CT_G2}.

\begin{theorem}[\textbf{Complex $\mathrm{G_2}$ constant term identity}]
\label{Thm_CT_G2_complex}
For $u,v\in\Complex$ such that $\Re(1+\tfrac{3}{2}u)>0$ and
$\Re(1+\tfrac{3}{2}(u+v))>0$,
\begin{multline}\label{Eq_CT_G2_complex}
\CT\bigg[
\Big(1-\frac{yz}{x^2}\Big)^u
\Big(1-\frac{xz}{y^2}\Big)^u
\Big(1-\frac{xy}{z^2}\Big)^u
\Big(1-\frac{x}{y}\Big)^v
\Big(1-\frac{y}{z}\Big)^v
\Big(1-\frac{z}{x}\Big)^v\,\bigg] \\
=\frac{\cos\big(\tfrac{1}{2}\pi u\big)
\cos\big(\tfrac{1}{2}\pi v\big)
\Gamma(1+\tfrac{3}{2}(u+v))\Gamma(1+\tfrac{3}{2}u)
\Gamma(1+u)\Gamma(1+v)}
{\Gamma(1+\tfrac{3}{2}u+v)\Gamma(1+u+\tfrac{1}{2}v)
\Gamma(1+\tfrac{1}{2}(u+v))\Gamma^2(1+\tfrac{1}{2}u)
\Gamma(1+\tfrac{1}{2}v)}.
\end{multline}
\end{theorem}

\begin{proof}
We adopt the method of proof employed by Habsieger and
Zeilberger \cite{Habsieger86,Zeilberger87}
in their proof of Theorem~\ref{Thm_HZ}.

If $A(x,y,z;a,u)$ denotes the kernel on the left of the
complex Morris identity \eqref{Eq_n3_b} for $n=3$, and if
and $G(x,y,z;u,v)$ denotes
the kernel on the left of \eqref{Eq_CT_G2_complex},
then
\[
G(x,y,z;u,v)=A(x/y,y/z,z/x,v,u).
\]
Therefore,
\begin{align*}
\CT G(x,y,z;u,v) &=\CT A(x/y,y/z,z/x;v,u) \\
&=\CT A(x,y,z;v,u)\big|_{xyz=1} \\
&=\sum_{b=0}^{\infty} \big[x^by^bz^b\big] A(x,y,z;v,u) \\
&=\cos\big(\tfrac{1}{2}\pi u\big)
\frac{\Gamma(1+\frac{3}{2}u)}{\Gamma^3(1+\frac{1}{2}u)}\,
\hyper{3}{2}{-v,-\frac{1}{2}u-v,-u-v}{1+\frac{1}{2}u,1+u}{1},
\end{align*}
where the last equality follows from \eqref{Eq_n3_b}.
Summing the $_3F_2$ series by Dixon's sum \eqref{Eq_Dixon}
with $(2a,b,c)\mapsto (-v,-\frac{1}{2}u-v,-u-v)$ completes the proof.
\end{proof}

Just as for the root system $\mathrm{A}_{n-1}$, we can use the complex
$\mathrm{G}_2$ constant term identity to proof a logarithmic analogue
of \eqref{Eq_CT_G2}.

\begin{theorem}\label{Thm_log_CT_G2}
Assume the representation of the $\mathrm{G}_2$ root system as
given in Section~\ref{Sec_2}, and let $\Phi^{+}_s$ and $\Phi^{+}_l$
denote the set of positive short and positive long roots respectively.
Define
\[
G(K,M)=\frac{1}{3}\,
\frac{(3K+3M)!!(3K)!!(2K)!!(2M)!!}{(3K+2M)!!(2K+M)!!(K+M)!!K!!K!!M!!}.
\]
Then for $k,m$ nonnegative integers,
\[
\CT\bigg[\,\eup^{-3\alpha_1-2\alpha_2}\log(1-\eup^{\alpha_2})
\prod_{\alpha\in\R_l^+} (1-\eup^{\alpha})
\prod_{\alpha\in\R_l} (1-\eup^{\alpha})^k
\prod_{\alpha\in\R_s} (1-\eup^{\alpha})^m \bigg]
=G(K,M),
\]
where $(K,M):=(2k+1,2m)$, and
\[
\CT\bigg[\,\eup^{-2\alpha_1-\alpha_2}\log(1-\eup^{\alpha_1})
\prod_{\alpha\in\R_s^+} (1-\eup^{\alpha})
\prod_{\alpha\in\R_l} (1-\eup^{\alpha})^k
\prod_{\alpha\in\R_s} (1-\eup^{\alpha})^m \bigg] \\
=G(K,M),
\]
where $(K,M):=(2k,2m+1)$.
\end{theorem}
We can more explicitly write the kernels of the two logarithmic
$\mathrm{G}_2$ constant term identities as
\begin{multline*}
\frac{z^2}{xy}
\Big(1-\frac{x^2}{yz}\Big)
\Big(1-\frac{y^2}{xz}\Big)
\Big(1-\frac{xy}{z^2}\Big)
\log\Big(1-\frac{y^2}{xz}\Big) \\ \times
\bigg(\Big(1-\frac{x^2}{yz}\Big)
\Big(1-\frac{y^2}{xz}\Big)
\Big(1-\frac{z^2}{xy}\Big)
\Big(1-\frac{yz}{x^2}\Big)
\Big(1-\frac{xz}{y^2}\Big)
\Big(1-\frac{xy}{z^2}\Big) \bigg)^k \\
\times
\bigg(\Big(1-\frac{x}{y}\Big)
\Big(1-\frac{x}{z}\Big)
\Big(1-\frac{y}{x}\Big)
\Big(1-\frac{y}{z}\Big)
\Big(1-\frac{z}{x}\Big)
\Big(1-\frac{z}{y}\Big)\bigg)^m
\end{multline*}
and
\begin{multline*}
\frac{z}{x}
\Big(1-\frac{x}{y}\Big)
\Big(1-\frac{y}{z}\Big)
\Big(1-\frac{x}{z}\Big)
\log\Big(1-\frac{x}{y}\Big) \\ \times
\bigg(\Big(1-\frac{x^2}{yz}\Big)
\Big(1-\frac{y^2}{xz}\Big)
\Big(1-\frac{z^2}{xy}\Big)
\Big(1-\frac{yz}{x^2}\Big)
\Big(1-\frac{xz}{y^2}\Big)
\Big(1-\frac{xy}{z^2}\Big) \bigg)^k \\
\times
\bigg(\Big(1-\frac{x}{y}\Big)
\Big(1-\frac{x}{z}\Big)
\Big(1-\frac{y}{x}\Big)
\Big(1-\frac{y}{z}\Big)
\Big(1-\frac{z}{x}\Big)
\Big(1-\frac{z}{y}\Big)\bigg)^m
\end{multline*}
respectively.

\begin{proof}[Proof of Theorem~\ref{Thm_log_CT_G2}]
If we differentiate \eqref{Eq_CT_G2_complex} with respect to $u$,
use the cyclic symmetry in $(x,y,z)$ of the kernel on the left, and finally
set $u=2k+1=K$, we get
\begin{multline*}
3 \CT\bigg[
\log\Big(1-\frac{xz}{y^2}\Big)
\Big(1-\frac{yz}{x^2}\Big)^K
\Big(1-\frac{xz}{y^2}\Big)^K
\Big(1-\frac{xy}{z^2}\Big)^K
\Big(1-\frac{x}{y}\Big)^v
\Big(1-\frac{y}{z}\Big)^v
\Big(1-\frac{z}{x}\Big)^v\,\bigg] \\
=(-1)^{k+1} \frac{\pi}{2} \,
\frac{\cos\big(\tfrac{1}{2}\pi v\big)
\Gamma(1+\tfrac{3}{2}(K+v))\Gamma(1+\tfrac{3}{2}K)
\Gamma(1+K)\Gamma(1+v)}
{\Gamma(1+\tfrac{3}{2}K+v)\Gamma(1+K+\tfrac{1}{2}v)
\Gamma(1+\tfrac{1}{2}(K+v))\Gamma^2(1+\tfrac{1}{2}K)
\Gamma(1+\tfrac{1}{2}v)}.
\end{multline*}
Setting $v=2m=M$ and carrying out some simplifications using
\eqref{Eq_Legendre} and \eqref{Eq_phfrac} completes the proof
of the first claim.

In much the same way, if we differentiate \eqref{Eq_CT_G2_complex}
with respect to $v$, use the cyclic symmetry in $(x,y,z)$ and
set $v=2m+1=M$, we get
\begin{multline*}
3 \CT\bigg[
\log\Big(1-\frac{x}{y}\Big)
\Big(1-\frac{yz}{x^2}\Big)^u
\Big(1-\frac{xz}{y^2}\Big)^u
\Big(1-\frac{xy}{z^2}\Big)^u
\Big(1-\frac{x}{y}\Big)^M
\Big(1-\frac{y}{z}\Big)^M
\Big(1-\frac{z}{x}\Big)^M\,\bigg] \\
=(-1)^{m+1} \frac{\pi}{2} \,
\frac{\cos\big(\tfrac{1}{2}\pi u\big)
\Gamma(1+\tfrac{3}{2}(u+M))\Gamma(1+\tfrac{3}{2}u)
\Gamma(1+u)\Gamma(1+M)}
{\Gamma(1+\tfrac{3}{2}u+M)\Gamma(1+u+\tfrac{1}{2}M)
\Gamma(1+\tfrac{1}{2}(u+M))\Gamma^2(1+\tfrac{1}{2}u)
\Gamma(1+\tfrac{1}{2}M)}.
\end{multline*}
Setting $u=2k=K$ yields the second claim.
\end{proof}

\section{Other root systems}\label{Sec_Other}

Although further root systems admit complex analogues of the Macdonald
constant term identities \eqref{Eq_CT_Macdonald} or
\eqref{Eq_CT_Macdonald_II}, it seems the existence of
elegant logarithmic identities is restricted to
$\mathrm{A}_{2n}$ and $\mathrm{G}_2$.
To see why this is so, we will discuss the root systems
$\mathrm{B}_n$, $\mathrm{C}_n$ and $\mathrm{D}_n$.
In order to treat all three simultaneously, it will be convenient to
consider the more general \emph{non-reduced} root system $\mathrm{BC}_n$.
With $\epsilon_i$ again denoting the $i$th standard unit vector in $\Real^n$,
this root system is given by
\[
\Phi=\{\pm\epsilon_i:~1\leq i\leq n\}\cup\{\pm2\epsilon_i:~1\leq i\leq n\}
\cup\{\pm\epsilon_i\pm \epsilon_j:~1\leq i<j\leq n\}.
\]
Using the Selberg integral, Macdonald proved that \cite{Macdonald82}
\begin{multline}\label{Eq_BCn}
\CT\bigg[\,\prod_{i=1}^n (1-x_i^{\pm})^a(1-x_i^{\pm 2})^b
\prod_{1\leq i<j\leq n}(1-x_i^{\pm} x_j^{\pm})^k\bigg] \\
=\prod_{i=0}^{n-1}
\frac{(k+ik)!(2a+2b+2ik)!(2b+2ik)!}
{k!(a+b+ik)!(b+ik)!(a+2b+(n+i-1)k)!},
\end{multline}
where $a,b,k$ are nonnegative integers and where we have adopted the
standard shorthand notation
$(1-x^{\pm}):=(1-x)(1-1/x)$,
$(1-x^{\pm 2}):=(1-x^2)(1-1/x^2)$,
$(1-x^{\pm}y^{\pm}):=(1-xy)(1-x/y)(1-y/x)(1-1/xy)$.
For $b=0$ the above identity is the $\mathrm{B}_n$ case of
\eqref{Eq_CT_Macdonald_II},
for $a=0$ it is the $\mathrm{C}_n$ case of \eqref{Eq_CT_Macdonald_II}
and for $a=b=0$ it is the $\mathrm{D}_n$ case of
\eqref{Eq_CT_Macdonald_II} (and also \eqref{Eq_CT_Macdonald}).

A first task in finding a complex analogue of \eqref{Eq_BCn} is
to fix signatures $\tau_{ij}$ and $\sigma_{ij}$ for $1\leq i,j\leq n$
such that
\begin{equation}\label{Eq_sigmatau}
\prod_{1\leq i<j\leq n}(1-x_i^{\pm} x_j^{\pm})
=\prod_{1\leq i<j\leq n}
\Big(1-\big(x_ix_j\big)^{\sigma_{ij}}\Big)^2
\Big(1-\Big(\frac{x_i}{x_j}\Big)^{\tau_{ij}}\Big)^2.
\end{equation}
This would allow the rewriting of \eqref{Eq_BCn} as
\begin{multline}\label{Eq_BCn_II}
\CT\bigg[\,\prod_{i=1}^n (1-x_i^{\pm})^a(1-x_i^{\pm 2})^b
\prod_{1\leq i<j\leq n}
\Big(1-\big(x_ix_j\big)^{\sigma_{ij}}\Big)^{2k}
\Big(1-\Big(\frac{x_i}{x_j}\Big)^{\tau_{ij}}\Big)^{2k}
\bigg] \\
=\prod_{i=0}^{n-1}
\frac{(k+ik)!(2a+2b+2ik)!(2b+2ik)!}
{k!(a+b+ik)!(b+ik)!(a+2b+(n+i-1)k)!},
\end{multline}
after which $2k$ can be replaced by the complex variable $u$.

In the following we abbreviate \eqref{Eq_sigmatau} as
$L(X)=R_{\sigma\tau}(X)$.
In order to satisfy this equation, we note that
for an arbitrary choice of the signatures $\sigma_{ij}$ and $\tau_{ij}$,
\begin{align*}
L(X)
&=\prod_{1\leq i<j\leq n}
\Big(1-\big(x_ix_j\big)^{\pm\sigma_{ij}}\Big)
\Big(1-\Big(\frac{x_i}{x_j}\Big)^{\pm \tau_{ij}}\Big) \\
&=\prod_{1\leq i<j\leq n}(x_ix_j)^{-\sigma_{ij}}
\Big(\frac{x_i}{x_j}\Big)^{-\tau_{ij}}
\Big(1-\big(x_ix_j\big)^{\sigma_{ij}}\Big)^2
\Big(1-\Big(\frac{x_i}{x_j}\Big)^{\tau_{ij}}\Big)^2 \\
&=R_{\sigma\tau}(X)
\prod_{i=1}^n x_i^{-\sum_{j>i}(\sigma_{ij}+\tau_{ij})-
\sum_{j<i}(\sigma_{ji}-\tau_{ji})}.
\end{align*}
We must therefore fix the $\sigma_{ij}$ and $\tau_{ij}$ such that
\begin{equation}\label{Eq_sigmatau_condition}
\sum_{j=i+1}^n(\sigma_{ij}+\tau_{ij})+
\sum_{j=1}^{i-1}(\sigma_{ji}-\tau_{ji})=0
\end{equation}
for all $1\leq i\leq n$.
If we sum this over all $i$ this gives
\[
0=\sum_{1\leq i<j\leq n}\sigma_{ij}\equiv \binom{n}{2} \pmod{2}.
\]
We thus conclude that a necessary condition for
\eqref{Eq_sigmatau_condition},
and hence \eqref{Eq_sigmatau}, to hold is that
$n\equiv 0,1 \pmod{4}$.
As we shall show next it is also a sufficient condition,
as there are many solutions to \eqref{Eq_sigmatau_condition}
for the above two congruence classes.
\begin{lemma}\label{Lem_sigmatau}
For $n\equiv 1\pmod{4}$ define $m:=(n-1)/2$ and $p:=m/2$.
If we choose $\tau_{ij}$ as in \eqref{Eq_tau} and
$\sigma_{ij}$, $1\le i<j\le n$, as
\begin{equation}\label{Eq_sigma1}
\sigma_{ij}=\begin{cases}
-1 & \text{if $p<j-i\leq 3p$}, \\
\ph 1 & \text{otherwise},
\end{cases}
\end{equation}
then \eqref{Eq_sigmatau_condition}, and thus \eqref{Eq_sigmatau}, is satisfied.
\end{lemma}

We can extend the definition of $\sigma_{ij}$ to all
$1\leq i,j\leq n$ by setting $\sigma_{ij}=-\sigma_{ji}$.
Then the matrix $\Sigma=(\sigma_{ij})_{1\leq i,j\leq n}$ is a skew-symmetric
Toeplitz matrix.
For example, for $n=5$ the above choice for the $\sigma_{ij}$ generates
\[
\Sigma=\begin{pmatrix}[r]
0 & 1&-1&-1&1 \\
-1& 0& 1&-1&-1 \\
 1&-1& 0&1&-1 \\
 1& 1&-1&0&1 \\
-1& 1& 1&-1&0
\end{pmatrix}.
\]

\begin{proof}[Proof of Lemma~\ref{Lem_sigmatau}]
Note that by Lemma~\ref{Lem_tau} we only need to prove that
\[
\sum_{j=i+1}^n\sigma_{ij}+\sum_{j=1}^{i-1}\sigma_{ji}=0.
\]
If for $1\leq j\leq i-1$ we define $\sigma_{i,n+j}:=-\sigma_{ij}=\sigma_{ji}$
then this becomes
\begin{equation}\label{Eq_nul}
\sum_{j=i+1}^{n+i-1}\sigma_{ij}=0.
\end{equation}
We now observe that $\sigma_{i+1,j+1}=\sigma_{i,j}$.
For $j<n$ or $j>n$ this follows immediately from \eqref{Eq_sigma1}.
For $j=n$ it follows from $\sigma_{1,i+1}=\sigma_{i,n}$, which again
follows from \eqref{Eq_sigma1} since $p<n-i\leq 3p$ is equivalent to
$p<i\leq 3p$.
Thanks to $\sigma_{i+1,j+1}=\sigma_{i,j}$
we only need to check \eqref{Eq_nul} for $i=1$. Then
\[
\sum_{j=2}^n\sigma_{ij}=\sum_{j=2}^{p+1}1-\sum_{j=p+2}^{3p+1}1+\sum_{j=3p+2}^n1
=n-4p-1=0. \qedhere
\]
\end{proof}

\begin{lemma}\label{Lem_sigmatau_II}
For $n\equiv 0\pmod{4}$ define $m:=(n-2)/2$.
If we choose $\tau_{ij}$ as in \eqref{Eq_tau} and
$\sigma_{ij}$ as
\[
\sigma_{ij}=\begin{cases}
\ph 1 & \text{if $i+j$ is even or $i+j=m+2$}, \\
-1 & \text{if $i+j$ is odd and $i+j\neq m+2$},
\end{cases}
\]
then \eqref{Eq_sigmatau_condition}, and thus \eqref{Eq_sigmatau}, is satisfied.
\end{lemma}

\begin{proof}
By a simple modification of Lemma~\ref{Lem_tau} it follows that for $n$ even
and $m=(n-2)/2$,
\[
\sum_{j=i+1}^n\tau_{ij}-\sum_{j=1}^{i-1} \tau_{ji}=
\begin{cases}
-1 & \text{if $1\leq i\leq m+1$}, \\
\ph 1 & \text{if $m+1<i<n$}.
\end{cases}
\]
Hence we must show that
\[
\sum_{j=i+1}^n\sigma_{ij}+\sum_{j=1}^{i-1}\sigma_{ji}=
\begin{cases}
\ph 1 & \text{if $1\leq i\leq m+1$}, \\
-1 & \text{if $m+1<i<n$}.
\end{cases}
\]
But this is obvious. The sum on the left is over $n-1$ terms, with
one more odd $i+j$ then even $i+j$. Hence, without the exceptional
condition on $i+j=m+2$, the sum would always be $-1$. To have $i+j=m+2$
as part of one of the two sums we must have $i\leq m+1$, in which case one
$-1$ is changed to a $+1$ leading to a sum of $+1$ instead of $-1$.
\end{proof}

Lemmas~\ref{Lem_sigmatau} and \ref{Lem_sigmatau_II} backed up with
numerical data for $n=4$ and $n=5$ suggest the following
generalisation of~\eqref{Eq_BCn_II}.

\begin{conjecture}[\textbf{Complex $\boldsymbol{\mathrm{BC}_n}$ constant term identity}]
\label{Conj_CT_BCn_complex}
Let $n\equiv \zeta\pmod{4}$ where $\zeta=0,1$, and let
$u\in\Complex$ such
that $\min\{\Re(1+2b+(n-1)u),\Re(1+\tfrac{1}{2}nu)\}>0$.
Assume that $\tau_{ij}$ and $\sigma_{ij}$ for
$1\leq i<j\leq n$ are signatures satisfying
\eqref{Eq_sigmatau_condition}.
Then there exists a polynomial $P_n(x)$, independent of $a$ and $b$, such that $P_n(1)=1$ and
\begin{align}\label{Eq_BCn-complex}
\CT&\bigg[\,\prod_{i=1}^n (1-x_i^{\pm})^a(1-x_i^{\pm 2})^b
\prod_{1\leq i<j\leq n}
\Big(1-\big(x_ix_j\big)^{\sigma_{ij}}\Big)^u
\Big(1-\Big(\frac{x_i}{x_j}\Big)^{\tau_{ij}}\Big)^u
\bigg] \\
&=x^{n-\zeta} P_n(x^2)\,
\frac{\Gamma(1+\tfrac{1}{2}nu)}
{\Gamma(1+\frac{1}{2}(n-1)u)\Gamma^n(1+\tfrac{1}{2}u)}
\prod_{i=1}^{n-1}
\frac{\Gamma(1+iu)}{\Gamma(1+(i-\frac{1}{2})u)}\notag \\
& \qquad \quad \times
\prod_{i=0}^{n-1}
\frac{(\tfrac{1}{2}+\tfrac{1}{2}iu)_{a+b}(\tfrac{1}{2}+\tfrac{1}{2}iu)_b}
{(1+\tfrac{1}{2}(n+i-1)u)_{a+2b}}
\notag
\end{align}
where $x=x(u):=\cos\big(\frac{1}{2}\pi u\big)$.
Trivially, $P_1(x)=1$.
Conjecturally,
\[
P_4(x)=1 \quad\text{and}\quad
P_5(x)=\frac{1}{15}(3+4x+8x^2).
\]
\end{conjecture}

We note that the $\mathrm{D}_4$ case of the conjecture, i.e.,
$a=b=0$ and $n=4$, is equivalent to the following new
hypergeometric multisum identity
\[
\sum
\prod_{1\leq i<j\leq 4}(-1)^{k_{ij}}\binom{u}{k_{ij}}
\binom{u}{m_{ij}}
=\cos^4(\tfrac{1}{2}\pi u)
\frac{\Gamma(1+u)\Gamma^2(1+2u)\Gamma(1+3u)}
{\Gamma^5(1+\tfrac{1}{2}u)\Gamma^2(1+\tfrac{3}{2}u)\Gamma(1+\tfrac{5}{2}u)},
\]
where the sum is over $\{k_{ij}\}_{1\leq i<j\leq 4}$ and
$\{m_{ij}\}_{1\leq i<j\leq 4}$ subject to the constraints
\begin{align*}
k_{12}-k_{13}-k_{14}+m_{12}+m_{13}-m_{14}&=0, \\
k_{12}-k_{23}+k_{24}-m_{12}+m_{23}-m_{24}&=0, \\
k_{13}-k_{23}+k_{34}+m_{13}-m_{23}-m_{34}&=0, \\
k_{14}+k_{24}-k_{34}-m_{14}+m_{24}-m_{34}&=0,
\end{align*}
or, equivalently,
\[
\sum_{\substack{1\leq i<j\leq 4 \\[1pt] i=p \text{ or } j=p}}
(\tau_{ij}k_{ij}+\sigma_{ij}m_{ij})=0 \qquad \text{for $1\leq p\leq 4$}.
\]

Unfortunately, from the point of view of logarithmic constant term
identities, \eqref{Eq_BCn-complex} is not good news.
On the right-hand side the exponent $n-\zeta$ of $x$ is too high relative
to the rank $n$ of the root system. (Compare with $m=(n-1)/2$ versus
$n-1$ for $\mathrm{A}_{n-1}$.)
If we write \eqref{Eq_BCn-complex} as $L_n(u)=R_n(u)$ and define $K:=2k+1$,
then due to the factor $x^{n-\zeta}$,
$R_n^{(j)}(K)=0$ for all $1\leq j<n-\zeta$. Much like the Morris case,
$R_n^{(n-\zeta)}(K)$ yields a ratio of products of double factorials:
\begin{multline*}
R_n^{(n-\zeta)}(K)=
(n-\zeta)! P_n(0) \frac{(nK)!!}{((n-1)K)!!(K!!)^n}
\prod_{i=1}^{n-1} \frac{(2iK)!!}{((2i-1)K)!!} \\ \times
\prod_{i=0}^{n-1} \frac{(2b+iK-1)!!(2a+2b+iK-1)!!((n+i-1)K)!!}
{(2a+4b+(n+i-1)K)!!(iK-1)!!(iK-1)!!}.
\end{multline*}
However, if we differentiate $L_n(u)$ as many as $n-\zeta$ times,
a large number of different types of logarithmic terms give a
nonvanishing contribution to $L_n^{(n-\zeta)}(K)$---unlike
type $\mathrm{A}$ where only terms with the same functional form
(corresponding to perfect matchings) survive the specialisation $u=K$.
For example, for $n=4$ terms such as
\begin{gather*}
\log^3\Big(1-\frac{x_1}{x_2}\Big)\log\Big(1-\frac{1}{x_2x_3}\Big), \\
\log^2\Big(1-\frac{x_1}{x_2}\Big)
\log(1-x_1x_2)\log\Big(1-\frac{1}{x_2x_3}\Big), \\
\log^2\Big(1-\frac{x_1}{x_2}\Big)\log^2(1-x_3x_4),
\end{gather*}
and many similar such terms, all give nonvanishing contributions.

\subsection*{Acknowledgements}
The authors have greatly benefited from Tony Guttmann's and Vivien Challis'
expertise in numerical computations.
The authors also thank Antun Milas and
the anonymous referee for very helpful remarks, leading
to the inclusion of Section~\ref{Sec_New}.

\bibliographystyle{amsplain}

\end{document}